\documentclass{article}

\usepackage{fullpage}
\usepackage{graphicx}
\usepackage{graphics}
\usepackage{epsfig}
\usepackage{enumerate}
\usepackage{soul}

\usepackage{subfigure}

\usepackage{amsmath}
\usepackage{amssymb}
\usepackage{amsthm}
\usepackage[english]{babel}
\usepackage[latin1]{inputenc}
\usepackage{latexsym}
\usepackage{pgf}
\usepackage{hyperref}

\newtheorem{hypothesis}{Hypothesis}
\newtheorem{proposition}{Proposition}
\newtheorem{lemma}{Lemma}

\newtheorem{remark}{Remark}

\def\ds{\displaystyle}

\def\Rset{\mathbb{R}}
\def\tr{\mbox{tr}}
\def\det{\mbox{det}}

\begin{document}

\title{About the chemostat model
with a lateral diffusive compartment}
\author{{\sc Mar\'ia Crespo and Alain Rapaport}\\
MISTEA, Univ. Montpellier, INRA, Montpellier SupAgro\\
2 pl. Viala 34060 Montpellier, France\\
E-mails: {\tt maria.crespo@umontpellier.fr,alain.rapaport@inra.fr}}
\date{\today}

\maketitle

\begin{abstract}
We consider the classical chemostat model with an additional
compartment connected by pure diffusion, and analyze its asymptotic
properties.
We investigate conditions under which this spatial structure is
beneficial for species survival and yield conversion, compared to
single chemostat. 
Moreover we look for the best structure (volume repartition and diffusion rate) which minimizes the volume required to attain a desired yield conversion. The analysis reveals that configurations with a single tank connected by
diffusion to the input stream can be the most efficient.

\noindent {\bf Key-words.} chemostat model, compartments, diffusion, stability, yield
conversion, optimal design.
\end{abstract}

\section{Introduction}\label{sec:intro}

The model of the chemostat has been developed as a mathematical
representation of the chemostat apparatus invented  in the fifities
simultaneously by Monod \cite{M50} and Novick \& Szilard \cite{NZ50},
for studying the culture of micro-organisms, and is still today of
primer importance (see for instance \cite{HH05,HRDL08,Wetal16}).
Its mathematical analysis has led to the so-called ``theory of the
chemostat'' \cite{HHW77,SW95,HLRS17}.
This model is widely used for industrial applications
with continuously fed ``bioreactors'' for fermentations \cite{R71,SWH16}
or waste-water treatments \cite{GDL99,HRDL08,Wetal16}, but also 
in ecology for studying populations of
micro-organisms (or plankton) in lakes, wetlands, rivers or
aquaculture ecosystems \cite{H67,BSNP73,J74,KK78,RV15,Cetal15}, 
The word ``chemostat'' is often used to describe continuous cultures
of micro-organisms, even though it can be quite far from the
the original experimental setup.
The classical model of the chemostat assumes a perfectly mixed media 
which is generally verified for small volumes. For
industrial bioreactors or natural lakes with large volumes, the
validity of this assumption becomes questionable. 
This is why several extensions of this model with
spatial considerations have been proposed and studied in the literature.

The classical approaches for modeling non ideally mixed chemostats or
bioreactors rely on a continuous representation of the spatial
dimension (with systems of p.d.e.~as in \cite{KB92,D04}) or on a finite
number of interconnected compartments with different flow conditions
(with systems of
o.d.e.~as in the ``general gradostat'' \cite{STW91,SF79,HRG11}). 
Although there are many works in the literature on p.d.e.~models for
unmixed bioreactors (where nutrient diffuses on the media from the
boundary of the domain, see for instance \cite{SW95,DS96}), there are
comparatively few works for bioreactors with an advection term that represents
an incoming nutrient which is pushed inward (as in the chemostat apparatus).
Moreover, most of the mathematical analysis available in the literature consider
a spatial heterogeneity only in the axial dimension of the bioreactors
(leading to 1d p.d.e.~or compartments connected in series) as in
tubular or ``plug-flow'' bioreactors
\cite{GPMW64,D99,DLHRM08,ZCD15,DZC17} and (simple) gradostats
\cite{LW79,LW81,T86,SW91}.
Surprisingly, configurations of tanks in parallel rather than in series
have been much less investigated, apart simple considerations
in chemical reaction engineering \cite{L99,F08}.

In many cases, the axial direction appears to
be the one that generates the larger heterogeneity between the input
and output (when the main current lines are along this axis),
especially for height and relatively thin tanks under significant flow
rate.
The modeling with compartments has become quite popular in
the optimal design of bio-processes \cite{VT91}, as it allows to
determine easily the optimal sizes of a given number of tanks in
series for minimizing the residence time
\cite{LT82,HR89,GBBT96,GBT96,HRT03,HRD04,HD05,NS06,RHM08}. 
Several studies have shown the huge benefit that can
be obtained from one to two tanks in series (and even more but
marginally less with more than two).
Such considerations are similar to patches models or {\em islands
  models}, commonly used in theoretical ecology
\cite{MW67,H99} (or lattice differential equations \cite{SV09}). 
For instance, a recent investigation studies the influence of these structures 
on a consumer/resource model \cite{GGLM10}. However, ecological
consumer/resource models are similar to chemostat models apart the
source terms that are modeled as constant intakes of nutrient, instead
of dilution rates (or Robin boundary conditions) that are rather met
in liquid media.

Recent mathematical studies have revealed that considering
heterogeneity in directions transverse to the axial one (with 2d or
3d p.d.e models or compartments models with non serial interconnections) 
could have a significant impacts on the performances of the bioreactor and the
input-output behavior \cite{CIRR17}.
From an operational view point, it is often reported that ``dead
zones'' are observed in bioreactors and that the effective volumes
of the tanks have to be corrected in the models to provide accurate
predictions \cite{L99,HWBWSZ07,GS92,RDQ94,RGNL95,VAO05,SSRLHD10}.
Segregated habitats are also considered in lakes, where the bottom 
can be modeled as a {\em dead zone} and nutrient mixing between the
two zones is achieved by diffusion rate \cite{NT06}.
In a similar way, stagnant zones are well-known to occur in porous
media such as soils, at various extents depending on soil structure.
The effect of these dead zones on reactive and conservative mass
transport, and thus in turn on the biogeochemical cycles of elements,
can also be significant \cite{TI08,SJM00}.
However there have been relatively few analysis of models with
explicit ``dead-zones''. The wording ``dead-zone'' might be
slightly miss-leading as it can make believe that a part of the tank 
where there is no advection stream from the input flow has no biological
activity. But this does not necessarily mean that these ``dead-zones''
are entirely disconnected from other parts of the reactor. It is
likely to be influenced by diffusion rather than convection. This is
why we prefer to qualify these zones as ``lateral-diffusive compartments''.

The aim of the present work is to analyze the
chemostat model with two compartments (or two tanks), one of them
being connected by ``lateral-diffusion'', and to investigate
conditions under which having this compartment could be beneficial 
for the yield conversion compared to the chemostat model with a single
compartment of the same total volume. Although this structure falls into a particular
case of the {\em general gradostat}, the existing results in the
literature (see, e.g., \cite{T86,SW95})
are too general to give accurate conditions for the existence of a
positive equilibrium and its stability, and do not compare the
performances of each configuration.
The structure considered here can be also seen as a limiting case of
the pattern ``chemostats in parallel with diffusion connection'' studied in \cite{HRG11},
with only one vessel receiving an input flow rate. Nevertheless, this later
reference imposes some restrictions such as linear reaction between species and removal rate large
enough to avoid washout with a single tank. In the present work, we conduct a
deeper model analysis and investigate the impact of lateral diffusion from two
view points:
\begin{enumerate}
\item From an ecological perspective, we study the effect of the
  diffusion on a given volume repartition in terms of resource
  conversion. In particular, we aim at characterizing situations for
  which having a structure with lateral diffusion is better than having a single perfectly mixed volume.
\item From an engineering perspective, we look for the best volume repartition which, for a given diffusion rate, minimizes the total volume required to attain a desired resource conversion. This allows us to revisit the optimal design problem with such configurations, that was previously tackled but considering tanks connected in series (see, e.g., \cite{GBBT96,HRT03,HD05}). Additionally, we aim at determining the diffusion rate parameter that gives the best volume reduction.
\end{enumerate}

The article is organized as follows: in Section \ref{sec:modeling} we introduce the model 
describing the dynamics in the chemostat composed of two compartments,
one of them being connected by diffusion, and give results for the 
nonnegativity and boundedness of its solutions. In Section \ref{sec:equilibria} we determine the steady states and 
analyze its asymptotic stability. Section \ref{sec:performances}
investigates the resource conversion rate and characterizes when this structure is better than the single chemostat (i.e. a single perflectly mixed volume). Section \ref{sec:design} is dedicated to optimal design questions,
firstly when the diffusion rate is fixed and then when it can be tuned. Finally, Section \ref{sec:Interpretation} discusses and interprets the results.

\section{Modeling and preliminaries results}\label{sec:modeling}
We consider configurations of one tank of volume
$V_{1}$ interconnected by Fickian diffusion with a tank of volume
$V_{2}$, as depicted on Figure \ref{figdeadzone}.
\begin{center}
\begin{figure}[h]
\begin{center}
\includegraphics[height=4cm]{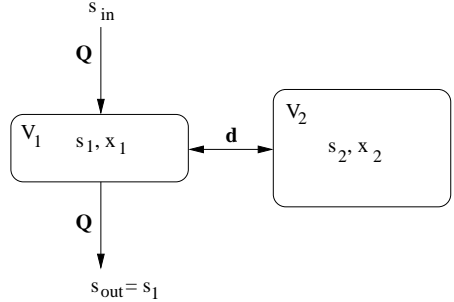}
\caption{Interconnection with lateral-diffusion.
\label{figdeadzone}}
\end{center}
\end{figure}
\end{center}
We denote by $s_{i}$, $x_{i}$ the concentrations of substrates and
biomass in tank $i=1, 2$ and write the equations of the chemostat model for
such interconnections:
\begin{equation}
\label{sys}
\left\{\begin{array}{lll}
\dot s_{1} & = & \ds -\mu(s_{1})x_{1} + \frac{Q}{V_{1}}(s_{\rm in}-s_{1})
+\frac{d}{V_{1}}(s_{2}-s_{1})\\[2mm]
\dot x_{1} & = & \ds \mu(s_{1})x_{1}-\frac{Q}{V_{1}}x_{1}+\frac{d}{V_{1}}(x_{2}-x_{1})\\[2mm]
\dot s_{2} & = & \ds -\mu(s_{2})x_{2} +\frac{d}{V_{2}}(s_{1}-s_{2})\\[2mm]
\dot x_{2} & = & \ds \mu(s_{2})x_{2}+\frac{d}{V_{2}}(x_{1}-x_{2})
\end{array}\right.
\end{equation}
where we have assumed, without any loss of generality, that the yield
conversion factor of substrate into biomass is equal to $1$.
The parameters $Q$ and $s_{\rm in}$ denote the flow rate and substrate
concentration of the input stream, while the parameter $d>0$ is the diffusion
coefficient  between the two tanks (that we assume
to be identical for the substrate and the micro-organisms).
The specific growth rate function of the micro-organisms is denoted
$\mu$ and fulfills the classical following assumption.

\begin{hypothesis}
\label{H1}
The growth function $\mu(\cdot)$ is increasing concave function with $\mu(0)=0$.
\end{hypothesis}

A typical such instance of function $\mu$ is given by the Monod law
(see, e.g., \cite{HLRS17,SW95}):
\[
\mu(s)=\mu_{max}\frac{s}{K+s},
\]
where $\mu_{max}$ is the maximum specific growth rate and $K$ is the half-saturation constant. 
\medskip
\begin{lemma}
\label{lem_inv}
The non-negative orthant $\Rset_{+}^4$ is invariant by dynamics \eqref{sys} and 
any solution in $\Rset_{+}^4$ is bounded.
\end{lemma}

\begin{proof}
Define $z_{i}=s_{\rm in}-s_{i}-x_{i}$ for each tank $i=1,2$ and consider
the dynamics \eqref{sys} in $(z,s)$ coordinates:
\begin{equation}
\label{sys_zs}
\left\{\begin{array}{lll}
\dot z_{1} & = & \ds -\frac{Q}{V_{1}}z_{1}-\frac{d}{V_{1}}(z_{1}-z_{2}),\\[2mm]
\dot s_{1} & = & \ds -\mu(s_{1})(s_{\rm in}-s_{1}-z_{1}) + \frac{Q}{V_{1}}(s_{\rm in}-s_{1})
+\frac{d}{V_{1}}(s_{2}-s_{1}),\\[2mm]
\dot z_{2} & = & \ds -\frac{Q}{V_{2}}(z_{2}-z_{1}),\\[2mm]
\dot s_{2} & = & \ds -\mu(s_{2})(s_{\rm in}-s_{2}-z_{2}) +\frac{d}{V_{2}}(s_{1}-s_{2}).
\end{array}\right.
\end{equation}
This system has a cascade structure with a first independent
sub-system linear in $z$
\begin{equation}
\label{sys_m}
\dot z = \underbrace{\left[\begin{array}{cc}
\ds -\frac{Q+d}{V_{1}} & \ds \frac{d}{V_{1}}\\[4mm]
\ds \frac{d}{V_{2}} & \ds -\frac{d}{V_{2}}
\end{array}\right]}_{\ds A} 
z,
\end{equation}
where one has
\[
\tr (A)=-\frac{Q+d}{V_{1}}-\frac{d}{V_{2}}<0  \quad \mbox{and}
\quad
\det (A)=\frac{Qd}{V_{1}V_{2}}>0.
\]
Therefore the matrix $A$ is Hurwitz and any solution $z$ of
\eqref{sys_m} converges exponentially to $0$.
Then, the solution $s$ can be written as the solution of the non autonomous dynamics
\begin{equation}
\label{sys_st}
\dot s=F(t,s)=\left[\begin{array}{c}
\ds \left(\frac{Q}{V_{1}}-\mu(s_{1})\right)(s_{\rm in}-s_{1}) 
+\frac{d}{V_{1}}(s_{2}-s_{1})+\mu(s_{1})z_{1}(t)\\
\ds
-\mu(s_{2})(s_{\rm in}-s_{2})+\frac{d}{V_{2}}(s_{1}-s_{2})+\mu(s_{2})z_{2}(t)
\end{array}\right].
\end{equation}
Notice that, for any $(t,s)$, one has
\[
\frac{\partial F_{1}(t,s)}{\partial s_{2}}=\frac{d}{V_{1}} > 0 \quad \mbox{and}
 \quad
\frac{\partial F_{2}(t,s)}{\partial s_{1}}=\frac{d}{V_{2}} > 0,
\]
and so the dynamics \eqref{sys_st} is cooperative (see, e.g., \cite{S95}).\\
Define $\check{F}_{1}(t,s):= -\frac{Q}{V_{1}}s_{1}-\mu(s_{1})(s_{\rm in}-s_{1}) 
+\frac{d}{V_{1}}(s_{2}-s_{1}) +\mu(s_{1})z_{1}(t)$, for which it follows that $F_{1}(t,s)>\check{F}_{1}(t,s)$ for any $(t,s)$. 
Proposition 2.1 in~\cite{S95} allows to
state that any solution of \eqref{sys_st} with $s_{i}(0)\geq 0$
($i=1,2$) satisfies $s_{i}(t)\geq \check{s}_{i}(t)$ ($i=1,2)$  for any
$t>0$, where $\check{s}$ is solution of the dynamics
\[
\dot{\check{s}}= \check{F}(t,\check{s})=\left[\begin{array}{l}
\check{F}_{1}(t,\check{s})\\
F_{2}(t,\check{s})
\end{array}\right] , \quad \check{s}(0)=0
\]
As one has $\check{F}(t,0)=0$ for any $t$, the solution $\check{s}$ is identically
null and one obtains that $s_{i}(t)$ ($i=1,2$) stays non-negative for
any positive $t$.

Similarly, $x$ can be written as a solution of a non-autonomous
cooperative dynamics
\[
\dot x = H(t,x)=\left[\begin{array}{c}
\ds \left(\mu(s_{1}(t))-\frac{Q+d}{V_{1}}\right)x_{1}+\frac{d}{V_{1}}x_{2}\\
\ds
\frac{d}{V_{2}}x_{1}+\left(\mu(s_{2}(t))-\frac{d}{V_{2}}\right)x_{2}
\end{array}\right]
\]
with $H(t,0)=0$, which allows to conclude that $x_{i}(t)$ ($i=1,2$)
stays non-negative for any positive $t$.

Finally, the convergence of $z$ to $0$ provides the boundedness of the
solutions $s(t)$, $x(t)$.
\end{proof}

\bigskip

Let us discuss the modeling of the two limiting cases that are not
covered by system \eqref{sys}:
\begin{itemize}
\item $V_{1}=0$. Physically, this corresponds to a single tank (of
  volume $V_{2}$) connected by diffusion to the input pipe with flow rate $Q$
  (see Figure \ref{figdilution}).
\begin{figure}[h]
\begin{center}
\includegraphics[height=3.5cm]{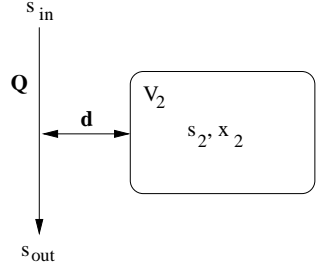}
\caption{Limiting case when $V_{1}=0$.
\label{figdilution}}
\end{center}
\end{figure}
There is no biological activity in the pipe but simply a dilution
given by the mass balance at the connection point:
\begin{equation}
\label{dilution}
\left\{\begin{array}{l}
Q(s_{\rm in}-s_{out})=d(s_{out}-s_{2}) \\
-Qx_{out}=d(x_{out}-x_{2})
\end{array}\right.
\quad \Rightarrow \quad
s_{out}=\frac{Qs_{\rm in}+ds_{2}}{Q+d} , \;
x_{out}=\frac{dx_{2}}{Q+d}
\end{equation}
Then the dynamics in the tank with $(s_{out},x_{out)}$ instead of
$(s_{1},x_{1})$ is given by the equations
\[
\left\{\begin{array}{lll}
\dot s_{2} & = & \ds -\mu(s_{2})x_{2}
+\frac{Qd}{(Q+d)V_{2}}(s_{\rm in}-s_{2})\\[4mm]
\dot x_{2} & = & \ds \mu(s_{2})x_{2}-\frac{Qd}{(Q+d)V_{2}}x_{2}
\end{array}\right.
\]
This is equivalent to have a single tank of volume $V_{2}$ with input flow
rate $Qd/(Q+d)$ but with an output given by $s_{out}=(Qs_{\rm in}+ds_{2})/(Q+d)$.
Equivalently, one can consider the system \eqref{sys} as a slow-fast
dynamics when the parameter $\epsilon=V_{1}$ is small. Then the fast
dynamics is
\[
\left\{\begin{array}{lll}
\epsilon \dot s_{1} & = & -\epsilon\mu(s_{1})x_{1}+Q(s_{\rm in}-s_{1})+d(s_{2}-s_{1})\\
\epsilon \dot x_{1} & = & \epsilon\mu(s_{1})x_{1}-Qx_{1}+d(x_{2}-x_{1})
\end{array}\right.
\]
and the slow manifold is given exactly by the system of equations
\eqref{dilution} with $(s_{1},x_{1})=(s_{out},x_{out})$.

\item $V_{2}=0$. The dynamics of $(s_{1},x_{1})$ is the one of the
  single chemostat model with volume $V_{1}$
\begin{equation}
\label{sysV1}
\left\{\begin{array}{lll}
\dot s_{1} & = & \ds -\mu(s_{1})x_{1} + \frac{Q}{V_{1}}(s_{\rm in}-s_{1})\\[2mm]
\dot x_{1} & = & \ds \mu(s_{1})x_{1}-\frac{Q}{V_{1}}x_{1}
\end{array}\right.
\end{equation}
This is also equivalent to having no diffusion ($d=0$) between the tanks.
\end{itemize}

\section{Study of the equilibria}\label{sec:equilibria}

For the analysis of the steady sates, it is convenient to introduce the function
\begin{equation}
\label{defbeta}
\beta(s)=\mu(s)(s_{\rm in}-s)
\end{equation}
that verifies the following property.
\begin{lemma}
\label{lemma1}
Under Hypothesis \ref{H1}, the function $\beta$
is strictly concave on $[0,s_{\rm in}]$. Thus, one can define the unique value
\begin{equation}\label{eq:shat}
\hat{s}= \arg \max_{s\in (0,s_{\rm in})} \beta(s).
\end{equation}
\end{lemma}

\begin{proof}
One has $\beta'(s)=\mu'(s)(s_{\rm in}-s)-\mu(s)$ and $\beta''(s)=\mu''(s)(s_{\rm in}-s)-2\mu'(s)$, which is negative for any $ s
  \in [0,s_{\rm in}]$.
\end{proof}

One can check that the {\em washout} state
$E^0=(0,s_{\rm in},0,s_{\rm in})^\top$ is always a steady-state of the
dynamics \eqref{sys}. In the next Propositions, we characterize the
existence of other equilibrium and their global stability.

\begin{proposition}
\label{prop-eq}
The washout equilibrium $E^0$ is the unique
steady state of \eqref{sys} exactly when $s_{\rm in}$ satisfies the condition
\begin{equation}
\label{cond-wo}
\mu(s_{\rm in}) \leq \frac{Q}{V_{1}} \mbox{ and } P(\mu(s_{\rm in})) \geq 0
\end{equation}
where $P$ is defined as
\begin{equation*}
P(X)=V_{1}V_{2}X^2-(dV_{1}+(Q+d)V_{2})X+dQ
\end{equation*}
When the condition \eqref{cond-wo} is not fulfilled, there exists an
unique positive steady state $E^{\star}$ of \eqref{sys} distinct from $E^0$.
\end{proposition}

\begin{proof}
From the two last equations of \eqref{sys}, one has
$s_{1}+x_{1}=s_{2}+x_{2}$ at steady-state, and from the two first
ones $s_{1}+x_{1}=s_{\rm in}$. The values $s_{1}$, $s_{2}$ at
steady state are solutions of the system of two equations
\begin{eqnarray}
\label{eq1}
0 & = & \left(\frac{Q}{V_{1}}-\mu(s_{1})\right)(s_{\rm in}-s_{1})
+\frac{d}{V_{1}}(s_{2}-s_{1})\\
\label{eq2}
0 & = & -\mu(s_{2})(s_{\rm in}-s_{2})+\frac{d}{V_{2}}(s_{1}-s_{2})
\end{eqnarray}
and $x_{1}$, $x_{2}$ at steady state are uniquely defined from each
solution $(s_{1},s_{2})$ of \eqref{eq1}-\eqref{eq2}.

Clearly $(s_{\rm in},s_{\rm in})$ is a solution of \eqref{eq1}-\eqref{eq2}.
We look for (positive) solutions different to $(s_{\rm in},s_{\rm in})$.

Posit
\begin{equation*}
\lambda_{1}(s_{\rm in}):=\max \left\{ s_{1} \in [0,s_{\rm in}] \, \vert \,
\mu(s_{1})\leq \frac{Q}{V_{1}} \right\}
\end{equation*}

From equations (\ref{eq1})-(\ref{eq2}), a solution different to
$(s_{\rm in},s_{\rm in})$ has to verify $s_{1}>s_{2}>0$ and then from equation
(\ref{eq1}), one has also $s_{1}<\lambda_{1}(s_{\rm in})$.
Define then the functions:
\begin{eqnarray*}
\phi_{1}(s_{1}) & := & s_{1}-\frac{Q-V_{1}\mu(s_{1})}{d}(s_{\rm in}-s_{1})
=s_{1}-\frac{Q}{d}(s_{\rm in}-s_{1})+\frac{V_{1}}{d}\beta(s_{1}),\\
\phi_{2}(s_{2}) & := & s_{2}+\frac{V_{2}\mu(s_{2})}{d}(s_{\rm in}-s_{2})
=s_{2}+\frac{V_{2}}{d}\beta(s_{2}).
\end{eqnarray*}
It is straightforward to see that any solution of (\ref{eq1})-(\ref{eq2}) fulfills
$s_{2}=\phi_{1}(s_{1})$ and $s_{1}=\phi_{2}(s_{2})$. One has
\begin{equation*}
\phi_{1}'(s_{1})=1+\frac{V_{1}}{d}\mu'(s_{1})(s_{\rm in}-s_{1})+\frac{Q-V_{1}\mu(s_{1})}{d}.
\end{equation*}
Therefore $\phi_{1}$ is increasing on $[0,\lambda_{1}(s_{\rm in})]$, with
$\phi_{1}(0)=-(Q/d)s_{\rm in}<0$ and
$\phi_{1}(\lambda_{1}(s_{\rm in}))=\lambda_{1}(s_{\rm in})>0$. Thus, $\phi_{1}$ is
invertible on $[-(Q/d)s_{\rm in},\lambda_{1}(s_{\rm in})]$ with
\begin{equation*}
\phi_{1}^{-1}(0) \in (0,\lambda_{1}(s_{\rm in})) .
\end{equation*}
From Lemma \ref{lemma1}, it follows that $\phi_{1}$ and $\phi_{2}$ are
strictly concave functions on $[0,s_{\rm in}]$.
Consider then the function
\begin{equation*}
\gamma(s_{2})=\phi_{2}(s_{2})-\phi_{1}^{-1}(s_{2}) \quad s_{2} \in [0,s_{\rm in}],
\end{equation*}
which is also strictly concave on $[0,s_{\rm in}]$.
Then, a solution $(s_{1},s_{2})$ can be written as a solution of
\begin{equation*}
\gamma(s_{2})=0 ,  \quad s_{1}=\phi_{2}(s_{2}) \quad \mbox{ with }
s_{2} \in [0,\lambda_{1}(s_{\rm in})].
\end{equation*}
Notice that one has $\gamma(s_{\rm in})=0$, and as $\gamma$ is strictly concave, it cannot have more than two zeros. Therefore there is at most one
solution $(s_{1},s_{2})$ different to $(s_{\rm in},s_{\rm in})$.
Furthermore, one has $\gamma(0)=-\phi_{1}^{-1}(0)<0$. Now, distinguish two different cases:

\medskip

$\bullet$ When $\lambda_{1}(s_{\rm in})<s_{\rm in}$ (or equivalently  $\mu(s_{\rm in})>Q/V_{1}$), one has
\begin{equation*}
\gamma(\lambda_{1}(s_{\rm in}))=\frac{QV_{2}}{dV_{1}}(s_{\rm in}-\lambda_{1}(s_{\rm in}))>0.
\end{equation*}
By using the Mean Value Theorem, one concludes that there exists $s_{2} \in
(0,\lambda_{1}(s_{\rm in}))$ such that $\gamma(s_{2})=0$.

\medskip

$\bullet$ When $\lambda_{1}(s_{\rm in})=s_{\rm in}$ (that is when $\mu(s_{\rm in})\leq Q/V_{1}$), 
the function $\gamma$ takes
positive value on the interval $[0,s_{\rm in}]$ if and only if
$\gamma'(s_{\rm in})<0$ ($\gamma$ being strictly concave on $[0,s_{\rm in}]$), 
or equivalently when the condition
\begin{equation*}
\phi_{2}'(s_{\rm in})<\frac{1}{\phi_{1}'(s_{\rm in})}
\end{equation*}
is fulfilled. Notice that one has $\phi_{1}'(s_{\rm in})>0$
because $\lambda_{1}(s_{\rm in})=s_{\rm in}$. So the condition can be
also written as
\begin{equation*}
\phi_{1}'(s_{\rm in})\phi_{2}'(s_{\rm in})<1 .
\end{equation*}
From the expressions of $\phi_{1}$ and $\phi_{2}$, one can write this
condition as
\begin{equation*}
\frac{(d+Q-V_{1}\mu(s_{\rm in}))(d-V_{2}\mu(s_{\rm in}))}{d^2}<1,
\end{equation*}
and easily check that this amounts to require $s_{\rm in}$ to satisfy
$P(\mu(s_{\rm in}))<0$.

\medskip We conclude that there exists a positive steady state if and
only if $\mu(s_{\rm in})>Q/V_{1}$ or $P(\mu(s_{\rm in}))<0$ and that this
steady state (when it exists) is unique.
\end{proof}

\begin{proposition}
\label{prop-stab}
When the washout equilibrium $E^0$ is the unique steady state, it is
globally asymptotically stable on $\Rset_{+}^4$.

When the positive steady state $E^\star$ exists, for any initial
condition except on a set of null measure, the solution of
\eqref{sys} converges asymptotically to $E^\star$, which is moreover
locally exponentially stable.
\end{proposition}

\begin{proof}
As shown in the proof of Lemma \ref{lem_inv}, the dynamics \eqref{sys} has a
cascade structure in $(z,s)$ coordinates, where $z$ converges
exponentially to $0$. Then $s$ converges to the set ${\mathcal
  S}=[0,s_{\rm in}]\times[0,s_{\rm in}]$ and is solution of the non-autonomous system
\eqref{sys_st} in the plane, which is asymptotically autonomous with
limiting dynamics
\begin{equation}
\label{sys_s}
\dot s=F_{a}(s)=\left[\begin{array}{c}
\ds \left(\frac{Q}{V_{1}}-\mu(s_{1}\right)(s_{\rm in}-s_{1}) 
+\frac{d}{V_{1}}(s_{2}-s_{1})\\
\ds
-\mu(s_{2})(s_{\rm in}-s_{2})+\frac{d}{V_{2}}(s_{1}-s_{2})
\end{array}\right]
\end{equation}
We study now the asymptotic behavior of dynamics \eqref{sys_s} in the
domain ${\cal S}$. Denote by ${\mathcal B}=\{s_{1}=s_{\rm in}\}\cup\{s_{2}=s_{\rm in}\}$ the subset of its boundary.
Accordingly to Proposition \ref{prop-eq}, this dynamics has
$s^{0}=(s_{\rm in},s_{\rm in})$ (which is the projection of $E^0$ on the $s$-plane) as
an equilibrium, and at most another equilibrium $s^{\star}$ (which is the
projection of $E^{\star}$ on the $s$-plane) that has to belong to ${\mathcal
  S}\setminus{\mathcal B}$.
On this last subset, we consider the
variables $\sigma_{i}=\log(s_{\rm in}-s_{i})$, whose dynamics are given by
\begin{equation}
\label{dynsigma}
\dot\sigma = \tilde F_{a}(\sigma)=
\left[\begin{array}{c}
\ds -\frac{Q}{V_{1}}+\mu(s_{\rm in}-e^{\sigma_{1}})-\frac{d}{V_{1}}\left(1-e^{\sigma_{2}-\sigma_{1}}\right)\\[4mm]
\ds \mu(s_{\rm in}-e^{\sigma_{2}})-\frac{d}{V_{2}}\left(1-e^{\sigma_{1}-\sigma_{2}}\right)
\end{array}\right]
\end{equation}
One obtains
\[
\mbox{div}(\tilde
F_{a})=-\mu'(s_{\rm in}-e^{\sigma_{1}})e^{\sigma_{1}}-\mu'(s_{\rm in}-e^{\sigma_{2}})e^{\sigma_{2}}-\frac{d}{V_{1}}e^{\sigma_{2}-\sigma_{1}}-\frac{d}{V_{2}}e^{\sigma_{1}-\sigma_{2}}
<0
\]
From Dulac criteria and Poincar\'e-Bendixon theorem (see for instance \cite{P01}), 
we conclude that bounded trajectories of (\ref{dynsigma}) cannot have limit cycle or closed path and necessarily converge to an equilibrium point.
Consequently, any trajectory of (\ref{sys_s}) in ${\mathcal S}$
either converges to the equilibrium $s^{\star}$ (if it exists)
or approaches the boundary ${\mathcal B}$. But
\[
\mbox{if there exists }t \mbox{ such that }\; s_{i}(t)=s_{\rm in} \, \mbox{ and }\, s_{j}(t)<s_{\rm in}, \; \mbox{ then } \;
\dot s_{i}(t)<0 \qquad (i\neq j)
\]
and so the only possibility for approaching ${\mathcal B}$ is to converge to $s^0$. 
This shows that the only 
non-empty closed connected invariant chain recurrent 
subsets of ${\mathcal S}$ are the isolated points $s^0$ and $s^{\star}$.\\

We use the theory of asymptotically autonomous systems 
(see Theorem 1.8 in \cite{MST95}) to deduce that any trajectory of system \eqref{sys} in $\Rset_{+}^4$ converges asymptotically to $E^0$ or
$E^{\star}$.\\

Due to the cascade structure of the dynamics \eqref{sys}
that is made explicit in the proof of Lemma \ref{lem_inv}, the Jacobian matrix
in the $(z,s)$ coordinates is
\[
J(s)=\left[\begin{array}{cc}
A & 0\\
\star & J_{a}(s)
\end{array}\right] \quad \mbox{with} \quad
J_{a}(s)=\left[\begin{array}{cc}
\ds -\frac{d}{V_{1}}\phi_{1}'(s_{1}) & \ds \frac{d}{V_{1}}\\[4mm]
\ds \frac{d}{V_{2}} & \ds  -\frac{d}{V_{2}}\phi_{2}'(s_{2})
\end{array}\right]
\]
where the matrix $A$ defined in \eqref{sys_m} is Hurwitz.
Thus, the eigenvalues of $J(s)$ are the ones of the matrix $A$ plus
the ones of $J_{a}(s)$.
One has
\[
\tr(J_{a}(s))=-d\left(\frac{\phi_{1}'(s_{1})}{V_{1}}+\frac{\phi_{2}'(s_{2})}{V_{2}}\right)
\quad \mbox{and} \quad
\det(J_{a}(s))=\frac{d^2}{V_{1}V_{2}}\left(\phi_{1}'(s_{1})\phi_{2}'(s_{2})-1\right)
.
\]

Accordingly to Proposition \ref{prop-eq}, the equilibrium $E^\star
\neq E^0$ exists when $P(\mu(s_{\rm in}))>0$ or $\mu(s_{\rm in})>Q/V_{1}$.

$\bullet$ When $P(\mu(s_{\rm in}))>0$, one has $\phi_{1}'(s_{\rm in})\phi_{2}'(s_{\rm in})<1$
or equivalently $\det(JF_{a}(s^0))<0$. Then $E^0$ is a saddle point
(with a stable manifold of dimension one) and we deduce that almost any
trajectory of \eqref{sys} converges to $E^\star$.

$\bullet$ When $\mu(s_{\rm in})>Q/V_{1}$, notice that the equilibrium $E^0$ is not necessarily
hyperbolic (as one can have $P(\mu(s_{\rm in}))=0$ which implies
  then $\det(JF_{a}(s^0))=0$) and we cannot conclude its stability properties directly. We
  setup a proof by contradiction, inspired by \cite{HRG11}.
Consider a solution of (\ref{sys_s}) with initial condition different
to $E^0$ that converges asymptotically to $E^0$, and
define the function
\[
v(t)=\min(\tilde x(t),x_{1}(t))
\mbox{ with } \tilde x=\frac{V_{1}x_{1}+V_{2}x_{2}}{V_{1}+V_{2}},
\]
which verifies $v(t)>0$ for any $t>0$ and has to converge to $0$ when
$t$ tends to $+\infty$.
The condition $\mu(s_{\rm in})>Q/V_{1}$ implies the existence of positive
numbers $T$ and $\eta$ such that $\mu(s_{1}(t))-Q/V_{1}>\eta$ for any $t>T$.
If $x_{1}(t) \leq x_{2}(t)$ with $t>T$, one has $v=x_{1}$ and
\[
\dot x_{1}\geq \left(\mu(s_{1}(t))-\frac{Q}{V_{1}}\right)x_{1}
\geq \eta x_{1} \geq 0
\]
If $x_{1}(t) \geq x_{2}(t)$ with $t>T$, one has $v=\tilde x$ and
\[
\dot{\tilde x} =
\frac{V_{1}}{V_{1}+V_{2}}\left(\mu(s_{1}(t))-\frac{Q}{V_{1}}\right)x_{1}
+ \frac{V_{2}}{V_{1}+V_{2}}\mu(s_{2}(t))x_{2}
\geq \frac{V_{1}}{V_{1}+V_{2}}\eta x_{1} \geq 0
\]
The positive function $v$ is thus non decreasing for $t>T$, in contradiction
with its convergence to $0$. We conclude
that any solution of \eqref{sys} with initial condition different to
$E^0$ converges to $E^\star$.\\

Finally, when the equilibrium $E^\star$ exists, the analysis conducted in the proof of Proposition \ref{prop-eq} allows us to deduce the inequalities $\phi_{1}'(s_{1}^\star)>0$ and $\gamma'(s_{2}^\star)>0$, which in turn imply $\phi_{1}'(s_{1}^\star)\phi_{2}'(s_{2}^\star)>1$  and so $\phi_{2}'(s_{2}^\star)>0$. Then, one has $\tr(J_{a}(s^\star))<0$ and 
$\det(J_{a}(s^\star))>0$ i.e. $J(s^\star)$ is Hurwitz, which proves
that the attractive equilibrium $E^\star$ is also locally exponentially stable.
\end{proof}

\section{Influence of lateral diffusion on the performances}\label{sec:performances}
Here, we investigate conditions under which having a second
compartment (as proposed in Section \ref{sec:modeling}) is beneficial
for the yield conversion compared to the chemostat model with a single
compartment of the same total volume. To this aim, we fix the hydric
volumes $V_{1}$ and $V_{2}$, the input flow $Q$, and analyze the
output map at steady state, as function of the diffusion parameter $d$. The benefits of the structured chemostat in terms of resource conversion are discussed in Section \ref{sec:Interpretation}.\\

Proposition \ref{prop-eq} defines properly the map $d \rightarrow
s_{1}^{\star}(d)$ for the unique non-trivial steady-state of system
\eqref{sys}, that we study here as a function of $d$. We start by deducing the range of existence of this steady-state.
\begin{proposition}\label{prop:s1dinit}
Let $V=V_{1}+V_{2}$ and $\bar{d}= V_{2}\mu(s_{\rm in}) \frac{Q-V_{1}\mu(s_{\rm in})}{Q-(V_{1}+V_2)\mu(s_{\rm in})}$. It follows that:  
\begin{enumerate}[(i)]
\item If $\mu(s_{\rm in}) < Q/V$, then the non-trivial equilibrium $s_{1}^{\star}(d)<s_{\rm in}$ exists when $d\in(0,\bar{d})$.
\item If $Q/V\leq \mu(s_{\rm in}) \leq Q/V_{1}$, then the non-trivial equilibrium $s_{1}^{\star}(d)<s_{\rm in}$ exists when $d>0$.
\item If $\mu(s_{\rm in}) > Q/V_{1}$, then the non-trivial equilibrium $s_{1}^{\star}(d)<s_{\rm in}$ exists when $d\geq 0$.
\end{enumerate}
\end{proposition}
\begin{proof}
When $d=0$ (that is, when the volume $V_{2}$ is detached) the classical equilibria analysis of the single chemostat model with volume $V_{1}$ (see, e.g., \cite{SW95}) assures that the positive equilibrium $s_{1}^{\star}$ exists when $\mu(s_{\rm in})>Q/V_{1}$, which corresponds to the case (iii) on the proposition statement. \\
When $d>0$, we prove cases (i)-(iii) by taking into account that they correspond to three different scenarios where condition \eqref{cond-wo} is not fulfilled. For ease of reasoning, we rewrite $P(\mu(s_{\rm in}))$ as
\begin{equation}\label{eq:Psin}
P(\mu(s_{\rm in}))=V_{2}\mu(s_{\rm in}) \underbrace{\big( V_{1} \mu(s_{\rm in}) - Q \big)}_{A} + d \underbrace{\big( Q - (V_{1}+V_{2})\mu(s_{\rm in}) \big)}_{B}.
\end{equation}
\textit{(i)}. In this case, the non-trivial equilibrium exists when $P(\mu(s_{\rm in}))<0$. It is straightforward to see that $A<0$, $B>0$ and so $s_{1}^{\star}(d)<s_{\rm in}$ exists when $0<d<\bar{d}=-V_{2}\mu(s_{\rm in})A/B$. 
\\
\textit{(ii)}. In this case, the non-trivial equilibrium exists when $P(\mu(s_{\rm in}))<0$. It is straightforward to see that $A<0$, $B<0$ and so $s_{1}^{\star}(d)<s_{\rm in}$ exists for all $d> 0$.
\\
\textit{(iii)}. In this case, the non-trivial equilibrium exists for all values of $P(\mu(s_{\rm in}))$, and so $s_{1}^{\star}(d)<s_{\rm in}$ exists for all $d\geq 0$. 
\end{proof}
We now study the two extreme situations: no diffusion and infinite diffusion.  
\begin{lemma}\label{lem:d0_inf}It follows that
\begin{enumerate}[(i)]
\item When $\mu(s_{\rm in})> Q/V_{1}$, the non trivial equilibrium of system \eqref{sys} fulfills 
$$ s_{1}^{\star}(0)= s_{1}^{\star,0},$$
where $s_{1}^{\star,0}= \mu^{-1}\left(\frac{Q}{V_{1}}\right)$ is the non-trivial steady state of the single chemostat model with volume $V_{1}$. \\
In other case $\lim_{d\rightarrow 0^{+}} s_{1}^{\star}(d)= s_{\rm in}$.
\item When $\mu(s_{\rm in})\geq Q/V$, the non trivial equilibrium of system \eqref{sys} fulfills 
$$\lim_{d\rightarrow +\infty} s_{1}^{\star}(d)= s_{1}^{\star,\infty},$$
where $s_{1}^{\star,\infty}= \mu^{-1}\left(\frac{Q}{V}\right),$ is the non-trivial steady state of the single chemostat model with volume $V= V_{1}+V_{2}$.
\end{enumerate}

\end{lemma}
\begin{proof} \mbox{}
\\
\textit{(i)}. This result is a direct consequence of the classical equilibria analysis of the single chemostat model with volume $V_{1}$ (see, e.g., \cite{SW95}), which assures that $s_{1}^{\star,0}$ exists when $\mu(s_{\rm in})>Q/V_{1}$.\\
\\
\textit{(ii)}. For any $d>0$, Proposition \ref{prop-eq} guarantees the existence of a unique non trivial equilibrium $s^{\star}=(s_{1}^{\star},s_{2}^{\star})\in (0,s_{\rm in})\times(0,s_{\rm in})$ that is solution of
\begin{equation}\label{eq12}
\left\{
\begin{array}{l}
d\big(s_{2}^{\star}-s_{1}^{\star}\big)= \big( V_{1}\mu(s_{1}^{\star})-Q \big)\big(s_{\rm in}-s_{1}^{\star}\big),\\
\\
d\big(s_{1}^{\star}-s_{2}^{\star}\big)= V_{2}\mu(s_{2}^{\star})\big(s_{\rm in}-s_{2}^{\star}\big).
\end{array}
\right.
\end{equation}
When $d$ is arbitrary large, one obtains
$$\lim_{d\rightarrow +\infty} s_{1}^{\star}-s_{2}^{\star}=0.$$
From equations \eqref{eq12}, one can also deduce the following equality (valid for any $d$)
\begin{equation}\label{eq:dinfty}
 (V_{1}\mu(s_{1}^{\star})-Q)(s_{\rm in}-s_{1}^{\star})=-V_{2}\mu(s_{2}^{\star})(s_{\rm in}-s_{2}^{\star}).
\end{equation}
Consequently, one has
$$\lim_{d\rightarrow +\infty} s_{1}^{\star}(d)= \lim_{d\rightarrow +\infty} s_{2}^{\star}(d)=s_{\rm in} \mbox{ or }\lim_{d\rightarrow +\infty} s_{1}^{\star}(d)= \lim_{d\rightarrow +\infty} s_{2}^{\star}(d)= s_{1}^{\star,\infty} \hspace{5mm} \mbox{as }V=V_{1}+V_{2},$$
where the classical equilibria analysis of the single chemostat model with volume $V$ assures that $s_{1}^{\star,\infty}$ exists when $\mu(s_{\rm in})>Q/V$. But Proposition \ref{prop:s1dinit} shows that, under the assumptions of the lemma, $s_{1}^{\star}(d)$ cannot converge to $s_{\rm in}$.
\end{proof}
We now present our main result concerning properties of the map $d
\rightarrow s_{1}^{\star}(d)$, defined at the non-trivial steady state. 
\begin{proposition}\label{prop:s1d}
Let $\hat{s}$ be defined in \eqref{eq:shat} and $V=V_{1}+V_{2}$. It follows that:  
\begin{enumerate}[(i)]
\item If $\mu(s_{\rm in}) < Q/V$, then the map $d \rightarrow s_{1}^{\star}(d)$ admits a minimum in $d^{\star}<\bar{d}$ that is strictly less than $s_{\rm in}$.
\item If $\mu(s_{\rm in}) \geq Q/V$ and $s_{1}^{\star,\infty}< \hat{s}$, then the map $d \rightarrow s_{1}^{\star}(d)$ admits a minimum in $d^{\star}<+\infty$, that is strictly less than $s_{1}^{\star,\infty}$.
\item If $\mu(s_{\rm in}) \geq Q/V$ and $s_{1}^{\star,\infty} \geq \hat{s}$, then the map $d \rightarrow s_{1}^{\star}(d)$ is decreasing and $s_{1}^{\star}(d)> s_{1}^{\star,\infty}$ for any $d> 0$.
\end{enumerate}
\end{proposition}

\begin{proof}
If one differentiates system \eqref{eq12} with respect to $d$, it follows that 
$$ \big(s_{2}^{\star}-s_{1}^{\star}\big)+ d\big(\partial_{d} s_{2}^{\star}- \partial_{d}s_{1}^{\star}\big)= \partial_{d}s_{1}^{\star} \underbrace{\big( Q + V_{1}\mu'(s_{1}^{\star})(s_{\rm in}-s_{1}^{\star}) - V_{1}\mu(s_{1}^{\star}) \big)}_{A},$$
$$ \big(s_{1}^{\star}-s_{2}^{\star}\big)+ d\big(\partial_{d} s_{1}^{\star}- \partial_{d}s_{2}^{\star}\big)= \partial_{d}s_{2}^{\star} \underbrace{\big( V_{2}\mu'(s_{2}^{\star})(s_{\rm in}-s_{2}^{\star}) - V_{2}\mu(s_{2}^{\star}) \big)}_{B},$$
which can be rewritten as
$$ \underbrace{\left [ \begin{array}{c c} A+d &-d \\ d & -B-d\end{array} \right]}_{\Gamma}\left ( \begin{array}{c} \partial_{d}s_{1}^{\star} \\ \partial_{d}s_{2}^{\star} \end{array}\right) = (s_{2}^{\star}-s_{1}^{\star}) \left ( \begin{array}{c} 1 \\ 1 \end{array}\right).$$
Remark that
$$ \begin{array}{r c l} A+d & = & d\phi_{1}'(s_{1}^{\star}), \\ B + d & = & d\phi_{2}'(s_{2}^{\star}), \\ {\rm det}(\Gamma)&=& d^{2}\big(1-\phi_{1}'(s_{1}^{\star})\phi_{2}'(s_{2}^{\star})\big). \end{array}$$
As seen in the proof of Proposition \ref{prop-stab}, one has that ${\rm det}(\Gamma)<0$ and so the derivatives $\partial_{d}s_{1}^{\star}$ and $\partial_{d}s_{2}^{\star}$ can be defined as
\begin{equation}
\begin{array}{r c l}
\partial_{d}s_{1}^{\star} & = & (s_{2}^{\star}-s_{1}^{\star}) \frac{-B}{{\rm det}(\Gamma)}, \\ \\ \partial_{d}s_{2}^{\star} & = & (s_{2}^{\star}-s_{1}^{\star}) \frac{A}{{\rm det}(\Gamma)}.
\end{array}
\end{equation}
Firstly, we prove that $A>0$ by showing that $\phi_{1}'(s_{1}^{\star}(d))>1$. \\
From Proposition \ref{prop-eq}, one has that the positive steady-state fulfills 
$$0< s_{1}^{\star}(d)< \lambda_{1}(s_{\rm in})= \min (s_{\rm in},s_{1}^{\star,0}).$$
Since $\phi_{1}$ is concave (equivalently, $\phi_{1}'$ is decreasing)
on $[0, \lambda_{1}(s_{\rm in})]$, one has that
$\phi_{1}'(s_{1}^{\star}(d)) > \phi_{1}'(\lambda_{1}(s_{\rm in}))$. Thus,
we prove that $A>0$ by showing $\phi_{1}'(\lambda_{1}(s_{\rm in}))>1$:
\begin{itemize} 
\item If $\mu(s_{\rm in})\leq Q/V_{1}$, then $\lambda_{1}(s_{\rm in})= s_{\rm in}$ and $\phi_{1}'(s_{\rm in})= 1 + \frac{Q-V_{1}\mu(s_{\rm in})}{d}>1$.
\item If $\mu(s_{\rm in})> Q/V_{1}$, then $\lambda_{1}(s_{\rm in})=
  s_{1}^{\star,0}$ and $\phi_{1}'(s_{1}^{\star,0})= 1 +
  \frac{V_{1}}{d}\mu'(s_{1}^{\star,0})(s_{\rm in}-s_{1}^{\star,0})>1$. 
\end{itemize}
Therefore one has $\partial_{d} s_{2}^\star >0$ i.e. $s_{2}^{\star}(\cdot)$ is an increasing map.\\
\\
Now, notice that $B= V_{2}\beta'(s_{2}^{\star}(d))$ and its sign depends on the relative position of $s_{2}^{\star}(d)$ with respect to parameter $\hat{s}$. The cases considered on the proposition statement are treated separately. 
\\
\\
\textit{(i)} Since $s_{2}^{\star}(\cdot)$ is increasing, $\lim_{d\rightarrow 0} s_{2}^{\star}(d)= 0$, $\lim_{d\rightarrow \bar{d}} s_{2}^{\star}(d)= s_{\rm in}$ and $\hat{s}\in (0,s_{\rm in})$, by using the Mean Value Theorem it follows that there exists a unique value $d\in (0, \bar{d})$ (denoted by $d^{\star}$) such that $s_{2}^{\star}(d^{\star})=\hat{s}$, with $\beta'(s_2^\star(d))>0$ for $d<d^\star$ and $<0$ for $d>d^\star$. Consequently, $\partial_{d}s_{1}^{\star}$ admits a unique minimum in $d^{\star}$, as $sgn(\partial_d s_1^{\star}(d))=-sgn(B)$.\\
\\
\textit{(ii)} Since $s_{2}^{\star}(\cdot)$ is increasing, $\lim_{d\rightarrow 0} s_{2}^{\star}(d)= 0$, $\lim_{d\rightarrow +\infty} s_{2}^{\star}(d)= s_{1}^{\star,\infty}$ and $\hat{s}\in (0,s_{1}^{\star,\infty})$, by using the Mean Value Theorem it follows that there exists a unique value $d>0$ (denoted by $d^{\star}$) such that $s_{2}^{\star}(d^{\star})=\hat{s}$. Consequently, $\partial_{d}s_{1}^{\star}$ admits a unique minimum in $d^{\star}$, with $s_1^\star(\cdot)$ decreasing on $[0,d^\star)$ and increasing on $(d^\star,+\infty)$. As $s_1^\star(\cdot)$ is increasing on $(d^\star,+\infty)$ and $\lim_{d\to+\infty} s_1^\star(d)=s_{1}^{\star,\infty}$ (from Lemma \ref{lem:d0_inf}), one necessarily has $s_1^\star(d^\star)< s_{1}^{\star,\infty}$. 
\\
\\
\textit{(iii)} Since $s_{2}^{\star}(\cdot)$ is increasing, $\lim_{d\rightarrow +\infty} s_{2}^{\star}(d)= s_{1}^{\star,\infty}$ and $\hat{s}>s_{1}^{\star,\infty}$, one has that $\beta'(s_{2}^{\star}(d))>0$, i.e., $s_{1}^{\star}(d)$ is decreasing for any $d>0$. As $\lim_{d\rightarrow +\infty}s_{1}^{\star}(d)= s_{1}^{\star,\infty}$, it follows that $s_{1}^{\star}(d) > s_{1}^{\star,\infty}$. 
\end{proof}

A schematic representation of the three situations depicted in Proposition \ref{prop:s1d} can be observed in Figures \ref{fig:s1d}-(a), \ref{fig:s1d}-(b) and \ref{fig:s1d}-(c), respectively.
\begin{figure}[ht!]
\centering
\subfigure[$\mu(s_{\rm in}) < \frac{Q}{V}$]{\includegraphics[width=0.3\textwidth]{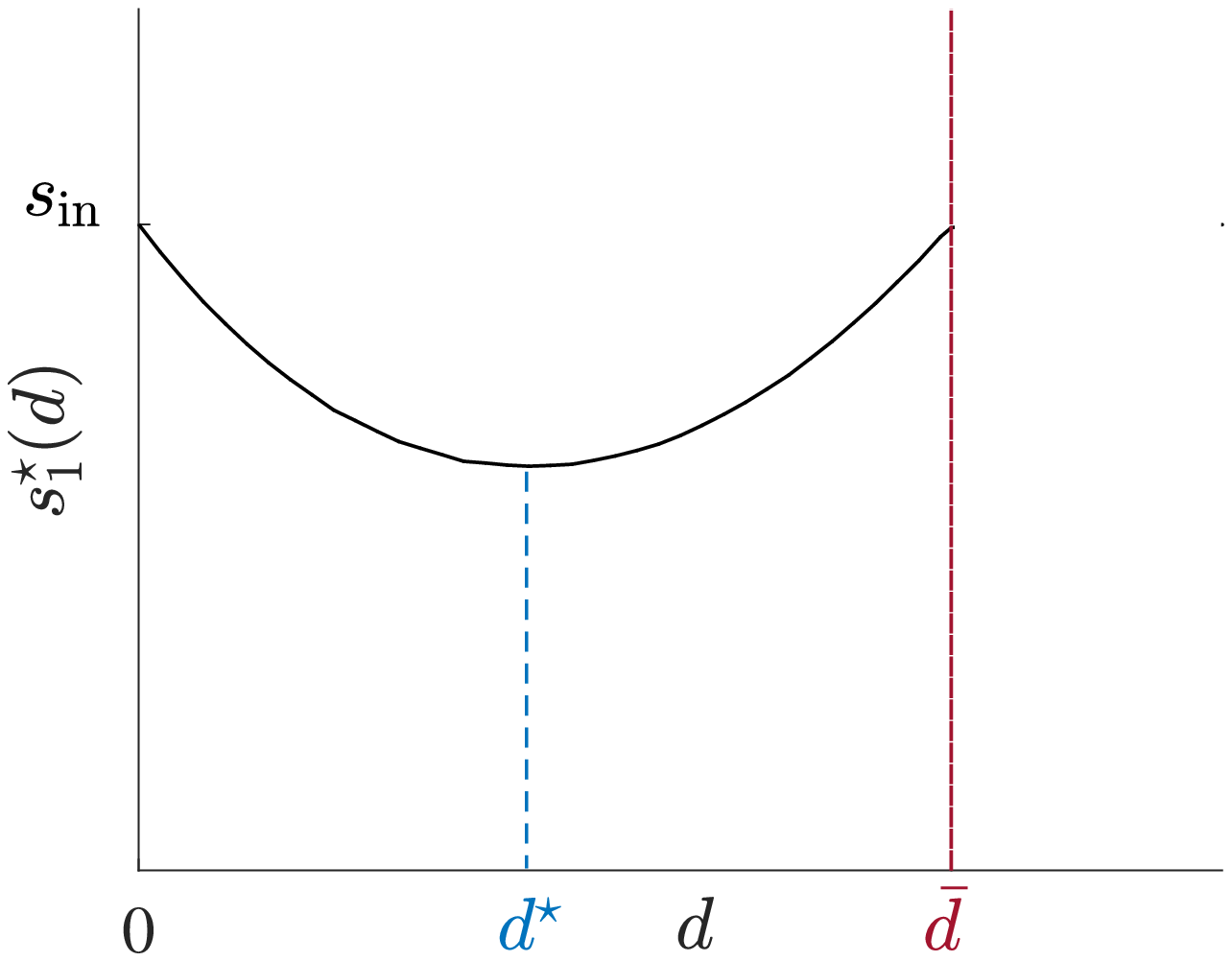}}
\subfigure[$\mu(s_{\rm in}) \geq Q/V$ and $s_{1}^{\star,\infty} < \hat{s}$]{\includegraphics[width=0.3\textwidth]{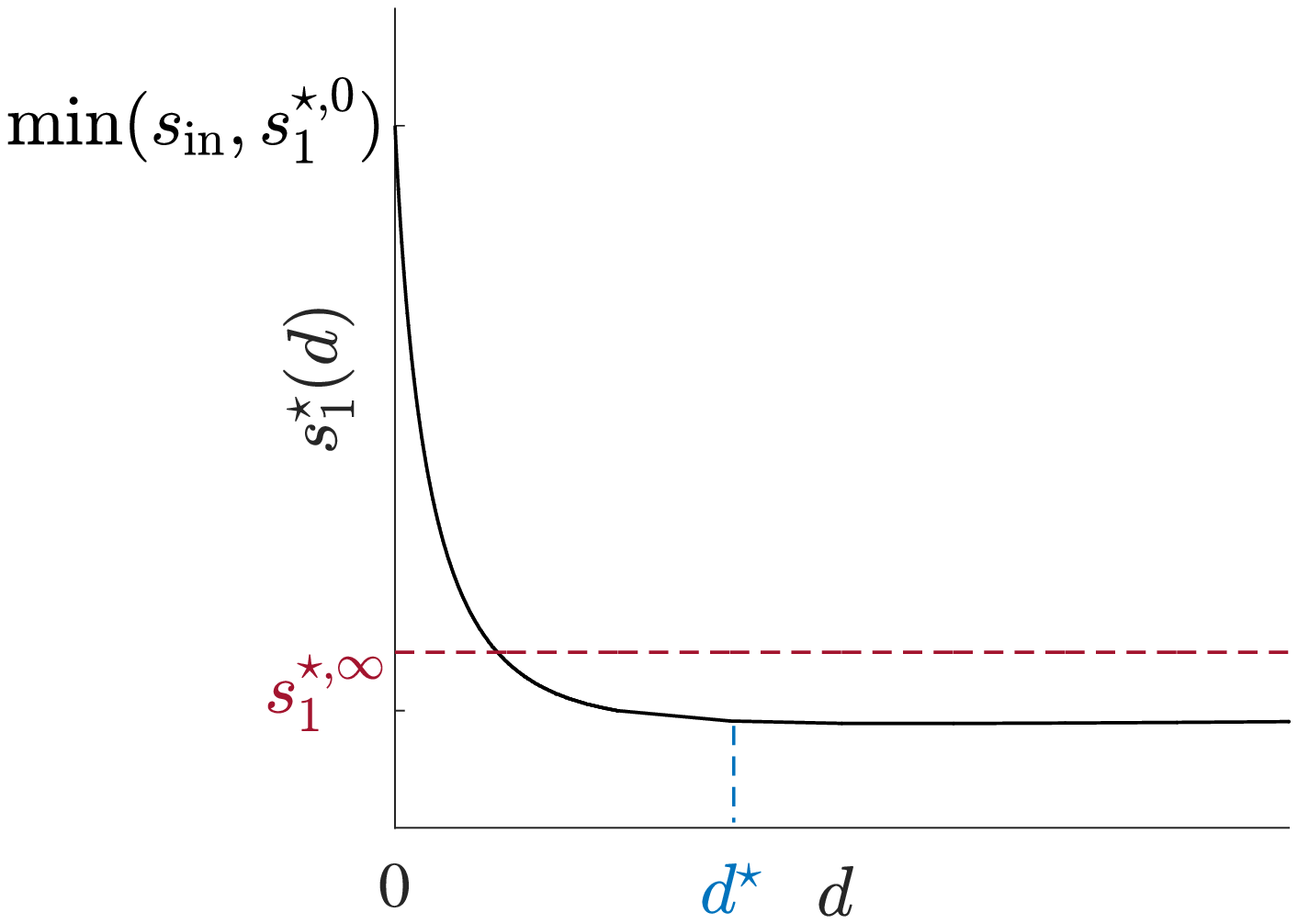}}
\subfigure[$\mu(s_{\rm in}) \geq Q/V$ and $s_{1}^{\star,\infty} \geq \hat{s}$]{\includegraphics[width=0.3\textwidth]{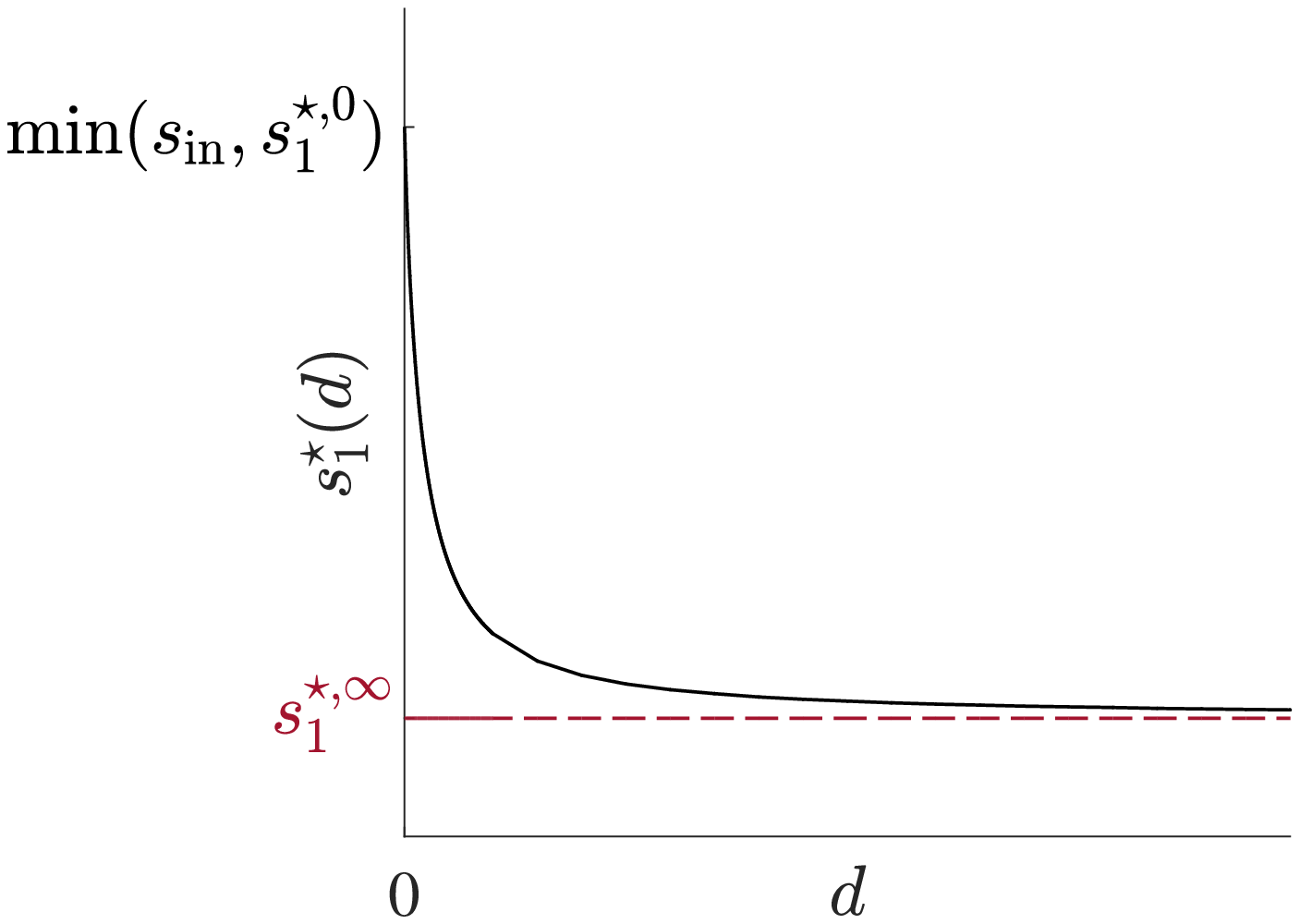}}
\caption{Graphical plot of the function $d \rightarrow s_{1}^{\star}(d)$.}
\label{fig:s1d}
\end{figure}

\section{Optimal configurations}\label{sec:design}
In this section, we optimize the main design parameters of the structured chemostat depicted on Figure \ref{figdeadzone} (reactor
volumes and diffusion rate) for minimizing the total volume, the output concentration being prescribed at steady state. 
One can easily check that minimizing the total volume is equivalent to maximizing the \textit{mean residence time} of the system, i.e., the mean time that a molecule spends in the chemostat (which affects its probability of reacting). A more detailed definition of the mean residence time, its measurement and interpretation can be found in Chapter 15 of \cite{nauman2002}.\\
This section is organized as follows: in Section \ref{sec:POptDgiven}, we solve the problem when the diffusion parameter is fixed. Then, in Section \ref{sec:POptDfree} we solve the full optimization problem in which the diffusion parameter is also considered as an optimization variable. Interpretations of the optimal results are presented in Section \ref{sec:Interpretation}.

\subsection{Parameter $d$ is fixed}\label{sec:POptDgiven}
Given a nominal desired value $s_{\rm ref} < s_{\rm in}$ as output of the process, we look for solutions of the optimization problem
\begin{equation}\label{eq:OptDgiven}
\min_{(V_1,V_{2})\in \Rset^{2}_{+}} \{ V_{1}+V_{2}\mbox{  : such that   } s_{1}= s_{\rm ref} \mbox{   at steady state }\},
\end{equation}
that we denote by $(V_{1}^{\rm opt},V_{2}^{\rm opt})$. \\
For the analysis of the solution of problem \eqref{eq:OptDgiven}, it
is convenient to introduce the functions
\begin{equation}
\label{defgG}
g(s)= \frac{1}{\beta(s)} \hspace{5mm} \mbox{and} \hspace*{5mm} G(s)=
\left(g(s_{\rm ref})-g(s)\right)(s - s_{\rm ref}),
\end{equation}
defined on $(0,s_{\rm in})$, where $\beta$ is defined in \eqref{defbeta}. Notice that function $g$ admits an unique
minimum at $\hat{s}$, by Lemma \ref{lemma1}, and satisfies $\lim_{s\rightarrow 0} g(s)=  \lim_{s\rightarrow +\infty} g(s) = +\infty$. \\

The solution to optimization problem \eqref{eq:OptDgiven} is given by the following proposition.
\begin{proposition}\label{prop:dgiven}
Define $\alpha= \max \big(0, s_{\rm ref} -\frac{Q}{d}(s_{\rm in}-s_{\rm ref})\big)$. The solution of problem \eqref{eq:OptDgiven} satisfies:
\begin{enumerate}[(i)]
\item If $\hat{s}\leq \alpha$, then $V_{1}^{\rm opt}= 0$ and $V_{2}^{\rm opt}= d g(\alpha)(s_{\rm ref}-\alpha)$.
\item If $\hat{s} \in (\alpha,s_{\rm ref})$, then $V_{1}^{\rm opt}= Q / \mu(s_{\rm ref}) +  d g(s_{\rm ref})(s_{2}^{\rm opt}- s_{\rm ref})$ and  $V_{2}^{\rm opt}= d g(s_{2}^{\rm opt})(s_{\rm ref}-s_{2}^{\rm opt})$, where 
$$s_{2}^{\rm opt}= \left\vert \begin{array}{l r} s_{G} & \mbox{if }\alpha\in [0,s_{G}], \\ \alpha & \mbox{if }\alpha\in (s_{G},\hat{s}), \end{array}\right.$$
$s_{G}$ being the unique minimum of the function $G$ on the interval $[\alpha,s_{\rm ref}]$. Moreover, $G'(s_{2}^{\rm opt})>0$ when $s_{2}^{\rm opt}=\alpha$.
\item If $\hat{s} \geq s_{\rm ref}$, then $V_{1}^{\rm opt}= Q / \mu(s_{\rm ref})$ and $V_{2}^{\rm opt}= 0$.
\end{enumerate}
\end{proposition}
\begin{remark}
From Proposition \ref{prop:dgiven}, one concludes that the particular configuration with $V_{1}=0$ (as the one depicted in Figure \ref{figdilution}) is optimal if $\hat{s}\leq \alpha$ or if $s_{G}<\alpha < \hat{s}<s_{\rm ref}$. 
\end{remark}
\begin{proof}[Proof of Proposition \ref{prop:dgiven}]
We replace the value of $s_{1}$ in system \eqref{eq1}-\eqref{eq2} by $s_{\rm ref}$
\begin{equation}\label{eq:1A1Dopt}
\left \{
\begin{array}{l}
0 = \frac{Q}{V_1}(s_{\rm in}-s_{\rm ref}) +  \frac{d}{V_1}(s_{2} - s_{\rm ref}) - \mu(s_{\rm ref})(s_{\rm in}-s_{\rm ref}) ,\\
\\
0 =  \frac{d}{V_{2}}(s_{\rm ref} - s_{2}) - \mu(s_{2})(s_{\rm in}-s_{2}).\\
\end{array}
\right.
\end{equation}
Considering function $g$, system \eqref{eq:1A1Dopt} can be written as
\begin{equation}\label{eq:v1v2s2}
\left \{
\begin{array}{l}
V_{1} = Qg(s_{\rm ref})(s_{\rm in}-s_{\rm ref}) +  d g(s_{\rm ref})(s_{2} - s_{\rm ref}):= v_{1}(s_{2}),\\
\\
V_{2} =  dg(s_{2})(s_{\rm ref}-s_{2}):=v_{2}(s_{2}).\\
\end{array}
\right.
\end{equation}
Thus, given model parameters $d$, $Q$, $s_{\rm in}$ and $s_{\rm ref}$, the volumes are completely characterized by variable $s_{2}$ and solving the optimization problem \eqref{eq:OptDgiven} is equivalent to look for solutions of the problem
\begin{equation}\label{eq:OptDgiven_s}
\min_{s_{2}\in \mathcal{S}_{2}}  \hspace*{5mm} v_{1}(s_{2}) + v_{2}(s_{2}),
\end{equation}
where $\mathcal{S}_{2}$ is the set of admissible values of
$s_{2}$. That is, the solution of problem \eqref{eq:OptDgiven} is
given by 
$\left(v_{1}(s_{2}^{\rm opt}),v_{2}(s_{2}^{\rm opt})\right)$, where $s_{2}^{\rm opt}$ is solution of problem \eqref{eq:OptDgiven_s}.\\
In order to determine the admissible set $\mathcal{S}_{2}$, we take into account that both volumes must be nonnegative and proceed as follows:
\begin{itemize}
\item $ v_{1}(s_{2})\geq 0 \Leftrightarrow  Qg(s_{\rm ref})(s_{\rm in}-s_{\rm ref}) +  d g(s_{\rm ref})(s_{2} - s_{\rm ref}) \geq 0 \Leftrightarrow s_{2} \geq s_{\rm ref} - \frac{Q}{d}(s_{\rm in}-s_{\rm ref}).$
\item $ v_{2}(s_{2})\geq 0 \Leftrightarrow  dg(s_{2})(s_{\rm ref}-s_{2})\geq 0 \Leftrightarrow s_{2}\leq s_{\rm ref}.$
\end{itemize}
Moreover, we impose variable $s_{2}$ to be nonnegative, since it describes a (substrate) concentration. One concludes that $\mathcal{S}_{2}= [\alpha,s_{\rm ref}]$. \\
\\
For analytical purposes, we rewrite problem \eqref{eq:OptDgiven_s} as
\begin{equation}\label{eq:OptDgiven_s2}
\min_{s_{2}\in [\alpha,s_{\rm ref}]} Q\underbrace{g(s_{\rm ref})(s_{\rm in}-s_{\rm ref})}_{A} + d G(s_{2}).
\end{equation}
The term $QA$ corresponds to the optimal volume obtained with a single tank, and with a view to reduce this value, we aim to characterize solutions of problem \eqref{eq:OptDgiven_s2} with function $G$ being negative. \\

The cases considered on the proposition statement are treated separately.
\begin{enumerate}[(i)]
\item $\hat{s}\leq \alpha$: Since function $g(\cdot)$ is increasing on the right of $\hat{s}$, then $g(s_{\rm ref})\geq g(s_{2})$ for all $s_{2} \in [\alpha,s_{\rm ref}]$. Consequently, function $G$ is negative on $[\alpha,s_{\rm ref}]$ and is minimized for $s_{2}^{\rm opt}=\alpha$.
\item $\hat{s}\in (\alpha,s_{\rm ref})$: In order to find $s_{2}^{\rm opt}$ on $[\alpha,s_{\rm ref}]$ such that $G(s_{2}^{opt})$ is minimum, we look for critical points of $G$, which satisfy 
$$ g'(s) = \frac{g(s_{\rm ref})-g(s)}{s-s_{\rm ref}}:=H(s).$$
By construction, function $g'$ is increasing on $(0,s_{\rm in})$,
$g'(\hat{s})=0$ (since $g$ is strictly convex, being equal to $1/\beta$ and
$\beta$ strictly concave by Lemma \ref{lemma1}), $g'(\cdot)<0$ on $(0,\hat{s})$ and $g'(\cdot)>0$ on $(\hat{s},s_{\rm in})$. Moreover, it is easy to see that the equation $H(s)=0$ has two solutions (and not more, as $g$ is strictlty convex): $s_{\rm ref}$ and $\bar{s}_{\rm ref}:= \{ s\in (0,\hat{s})\mbox{: }g(\bar{s}_{\rm ref})=g(s_{\rm ref})\}$. In addition, it follows that $H(\cdot)>0$ on $[0,\bar{s}_{\rm ref})$ and $H(\cdot)<0$ on $(\bar{s}_{\rm ref},s_{\rm ref})$. As a result, we can state that any critical point of function $G$ belongs to the interval $(\bar{s}_{\rm ref},\hat{s})$. \\
We show that there exist a unique critical point $s_{G}\in (\bar{s}_{\rm ref},\hat{s})$ of $G$ by proving that function $H$ is decreasing on this interval
$$H'(s)= \frac{-g'(s)(s-s_{\rm ref})- \big( g(s_{\rm ref})-g(s) \big)}{(s-s_{\rm ref})^{2}} = - \frac{g'(s)+H(s)}{s-s_{\rm ref}} < 0.$$
Graphically, the critical point $s_{G}$ is the abscissa of the intersection of the graphs $g'$ and $H$ (see Figure \ref{fig:CriticalPoints}). 
\begin{figure}[ht!]
\centering
\includegraphics[scale=0.4]{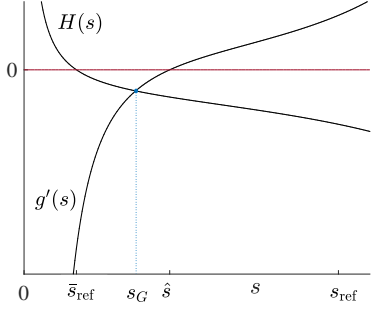}
\caption{Graphical determination of $s_{G}$.}
\label{fig:CriticalPoints}
\end{figure} 
Since we look for the minimum value of function $G$ on the interval $[\alpha,s_{\rm ref}]$, one has that $s_{2}^{\rm opt}$ depends on the value of $\alpha$. A direct conclusion is that $G'(s_{2}^{\rm opt})=0$ when $s_{2}^{\rm opt}=s_{G}$, while $G'(s_{2}^{\rm opt})>0$ when $s_{2}^{\rm opt}=\alpha$. 
\item $s_{\rm ref}\leq \hat{s}$: One has that $s_{2}\leq s_{\rm
    ref}\leq \hat{s}$ for all $s_{2} \in [\alpha,s_{\rm ref}]$. Since
  function $g(\cdot)$ is decreasing on the left of $\hat{s}$, then
  $g(s_{\rm ref})\leq g(s_{2})$. Consequently, function $G$ is
  nonnegative on $[\alpha,s_{\rm ref}]$ and the optimal value which
  makes it equal to zero is $s_{2}^{\rm opt}=s_{\rm ref}$. 
\end{enumerate}
\end{proof} 
\subsection{Characterization of the best value of the parameter $d$}\label{sec:POptDfree}
Given a nominal desired value $s_{\rm ref} < s_{\rm in}$ as output of the process, we look for solutions of the optimization problem
\begin{equation}\label{eq:OptDfree}
\min_{(V_1,V_{2},d)\in \Rset^{3}_{+}} \{ V_{1}+V_{2}\mbox{  : such that   } s_{1}= s_{\rm ref} \mbox{   at steady state }\},
\end{equation}
that we denote by $(V_{1}^{\ast},V_{2}^{\ast},d^{\ast})$. \\
\begin{proposition}\label{prop:OptDfree}
The solution of problem \eqref{eq:OptDfree} satisfies:
\begin{enumerate}[(i)]
\item If $\hat{s}< s_{\rm ref}$, then $V_{1}^{\ast}=0$, $V_{2}^{\ast}= Q (s_{\rm in}-s_{\rm ref}) g(\hat{s})$ and $d^{\ast}= Q \frac{s_{\rm in}-s_{\rm ref}}{s_{\rm ref}-\hat{s}}$. 
\item If $\hat{s}\geq s_{\rm ref}$, then $V_{1}^{\ast}= Q / \mu(s_{\rm ref})$, $V_{2}^{\ast}= 0$ and $d^{\ast}$ can take any value on the interval $[0,+\infty)$.
\end{enumerate}
\end{proposition}
\begin{proof}
In order to solve problem \eqref{eq:OptDfree}, we rely on the optimization results obtained in Section \ref{sec:POptDgiven}. Thus, $(V_{1}^{\ast},V_{2}^{\ast},d^{\ast})= (V_{1}^{\rm opt}(d^{\ast}),V_{2}^{\rm opt}(d^{\ast}),d^{\ast})$, where $d^{\ast}$ minimizes $V_1^{\rm opt}(d)+V_{2}^{\rm opt}(d)$ and $V_1^{\rm opt}$, $V_2^{\rm opt}$ are given by Proposition \ref{prop:dgiven}.\\
\\
\textit{(i)} From Proposition \ref{prop:dgiven}, one easily deduces that the total volume $V^{\rm opt}(d)=V_{1}^{\rm opt}(d)+V_{2}^{\rm opt}(d)$ fulfills
\begin{equation*}\label{eq:Vd}
V^{\rm opt}(d)= \left \vert \begin{array}{l c c} 
\frac{Q }{\mu(s_{\rm ref})} +  d G(s^{\rm opt}(d)) & \hspace*{2mm} \mbox{if } 0 \leq d < d^{\ast} & (\mbox{case (ii) in Prop.\ref{prop:dgiven}}), \\ \\
Q (s_{\rm in}-s_{\rm ref}) g(s_{\rm ref}-\frac{Q}{d}(s_{\rm in}-s_{\rm ref})) & \hspace*{2mm} \mbox{if } d\geq d^{\ast} & (\mbox{case (i) in Prop.\ref{prop:dgiven}}),
 \end{array}\right.
\end{equation*}
where $s^{\rm opt}$ must be now seen as a function of parameter $d$.\\
We analyze the monotonicity of function $V^{\rm opt}$. 
\begin{itemize}
\item When $0 \leq d < d^{\ast}$, one has that
$$\frac{\partial V^{\rm opt}}{\partial d} = G(s^{\rm opt}(d)) - d \frac{\partial G}{\partial s}|_{s=s^{\rm opt}(d)}\frac{\partial s^{\rm opt}(d)}{\partial d}.$$
From Proposition \ref{prop:dgiven}, it follows that $G(s^{\rm
  opt}(d))<0$ and $s^{\rm opt}(d)$ corresponds either to $s^{G}$ (with
$G'(s^{G})=0$) or to $\alpha$ (with $G'(\alpha)>0$). In both cases one
has $\frac{\partial V^{\rm opt}}{\partial d}<0$, that is, $V^{\rm opt}$ is decreasing on $[0,d^\ast)$. 
\item When $d\geq d^{\ast}$, one has that
$$\frac{\partial V^{\rm opt}}{\partial d} = \frac{Q^{2}}{2d^{2}}(s_{\rm in}-s_{\rm ref})^{2} g'(s_{\rm ref}-\frac{Q}{d}(s_{\rm in}-s_{\rm ref})).$$
By definition, $\hat{s}$ is the only value satisfying $g'(\hat{s})=0$ and so $d^{\ast}$ is the only critical point of function $V^{\rm opt}(\cdot)$. \\
It remains to prove that $d^{\ast}$ is a minimum of function $V^{\rm opt}(d)$. But 
$$\frac{\partial^{2} V^{\rm opt}}{\partial d^{2}} (d^{\ast})= \frac{Q^{3}}{4(d^{\ast})^{4}}(s_{\rm in}-s_{\rm ref})^{3} g''(\hat{s}),$$
which is positive as $g$ is strictly convex. Therefore $V^{\rm opt}$ is increasing on $[d^\ast,\infty)$.  
\end{itemize}
From these two points we conclude that the optimal value of $d$ is $d^\ast$.\\
\\
\textit{(ii)} This is a direct consequence of the statement (iii) in Proposition \ref{prop:dgiven}, since in this case the optimal volumes solution of problem \eqref{eq:OptDgiven} do not depend on parameter $d$. 
\end{proof}

\section{Discussion and interpretation of the
  results}\label{sec:Interpretation}
Here, we discuss the impact of the lateral diffusion from ecological and engineering points of view. Sections \ref{sec:IntepretationEcological} and \ref{sec:InterpretationEngineering} give a general interpretation of the results, while Section \ref{sec:NumericalExample} aims to quantify the benefits of the lateral diffusion in a particular numerical case.

\subsection{From an ecological view point}\label{sec:IntepretationEcological}
In Section \ref{sec:performances}, we have investigated the yield
conversion of the proposed structured chemostat and compared it with
the one of a single-tank chemostat. Our main result, presented in
Proposition \ref{prop:s1d} can be interpreted depending on the
global removal rate $D=Q/V$ and a threshold $\hat s$ (that is defined as the
maximizer of the function $\beta$ defined in \eqref{defbeta}) as follows:
\begin{enumerate}
\item If $D > \mu(s_{\rm in})$, a spatial distribution of the total
  volume $V$ could avoid the extinction of the micro-organisms while
  it happens when the volume $V$ is perfectly mixed.
Therefore, the lateral-diffusive compartment plays the role of a ``refuge'' for the
micro-organisms in case of large removal rate.
\item If $D \in [\mu(\hat{s}),\mu(s_{\rm in})]$, a spatial distribution of
  the total volume $V$ makes systematically increasing the output substrate concentration obtained when the volume $V$ is perfectly mixed.
\item If $D < \mu(\hat{s})$, a
  spatial distribution of the total volume $V$ could reduce the output
  substrate concentration obtained when the volume $V$ is perfectly
  mixed, but this is not systematic. This means that for small
  removal rates $D$ (as often met in soil ecosystems) one cannot
  know if a perfectly mixed model is under- or over-estimating the
  expected output level of the resource. 
\end{enumerate}
We have also analyzed the influence of the diffusion parameter
$d$ on the yield conversion in cases 1 and 3 and shown the existence of
a most efficient value $d^\star$.
The fact that a lateral-diffusive compartment is beneficial for
``extreme'' cases (i.e. large or small removal rates) does not appear to be an intuitive result for us.

\subsection{From an engineering view point}\label{sec:InterpretationEngineering}

In Section \ref{sec:design}, we studied optimal choices of the main
design parameters (reactor volume and diffusion rate) that minimize
the required volume for a given conversion rate. Our main results,
presented in Propositions \ref{prop:dgiven} and \ref{prop:OptDfree}
state that, when the desired substrate output concentration is above
certain threshold (more precisely, when $s_{\rm ref}>\hat{s}$), the volume
of a single-tank chemostat can be reduced by using the
structure with lateral diffusion. This result complements the work in \cite{LT82,GBBT96,HD05,NS06,ZCD15}, where the authors propose a methodology to diminish the volume of a single-tank chemostat when $s_{\rm ref}\leq \hat{s}$, by using either $N$ CSTR (continuous stirred tank reactor) in series or a CSTR connected in series to a PFR (plug flow reactor). We distinguish between the following cases:
\begin{itemize}
\item Diffusion coefficient is fixed. Depending on model
  parameters $s_{\rm in}$, $s_{\rm ref}$, $Q$, $d$ and $\mu(\cdot)$, the
  optimal structure may be composed of two tanks (of non null volumes
  $V_{1}^{\rm opt}$ and $V_{2}^{\rm opt}$) or a a single lateral tank
  (of volume $V_{2}^{\rm opt}$) connected by diffusion to the main stream. 
\item Diffusion coefficient can be optimized as well. The
  optimal structure is necessarily a single lateral tank (of volume
  $V_{2}^{\rm opt}$) connected by diffusion (with optimal diffusion
  rate $d^{\ast}= Q \frac{s_{\rm in}-s_{\rm ref}}{s_{\rm ref}-\hat{s}}$) to the
  main stream.
\end{itemize}
So an important message of this study is that the particular structure of
a single tank connected by diffusion to a pipe that conducts the input stream, as
depicted on Figure \ref{figdilution}, can be an efficient
configuration, better than a single tank directly under the main
stream. To our knowledge, this result is new in the literature.\\

The mathematical analysis has also revealed that the function $g$, i.e. 
the inverse of the function $\beta$ defined in \eqref{defbeta}, is playing an
important role in determining if the best configuration is composed of one or two tanks (more precisely the relative position of the output
reference value $s_{\rm ref}$ with respect to the minimizer $\hat s$ of
$g$). This is the same function than the one used for the optimal design of
tanks in series (with also a discussion on the relative position of
$s_{\rm ref}$  with respect to $\hat s$, see, e.g., \cite{GBBT96,HRT03}), but with two main differences:
\begin{enumerate}
\item Due to the particular considered structure, there is a trichotomy (one
  single mixed tank, two tanks, or one single lateral tank) instead of
  the dichotomy (one or more tanks) found for the problem with tanks in series. This trichotomy is discussed below with the help of the additional parameter $\alpha=\max(0,s_{\rm ref}-\frac{Q}{V}(s_{\rm in}-s_{\rm ref}))$.
\item For small values of $s_{\rm ref}$ (compared to $\hat s$), a
  lateral-diffusion compartment does not bring any improvement
  compared to a single perfectly mixed tank, while this is the
  opposite for tanks in series (i.e. several tanks are better than a
  single one when $s_{\rm ref}<\hat s$).
\end{enumerate}
These points can be grasped by the following graphical
interpretation. Consider the total volume $V$ required to obtain the
output concentration $s_{\rm ref}$ at steady state. In our case, it can be written in
terms of the function $g$ as follows
\begin{equation}\label{eq:Vdiscussion}
V=Q\underbrace{g(s_{\rm ref})(s_{\rm in}-s_{\rm ref})}_{A} + d\underbrace{(g(s_{\rm ref})-g(s_{2}^\star))(s_{2}^\star-s_{\rm ref})}_{B}
\end{equation}
where $s_{2}^\star$ is the steady state in the second compartment. One
can notice that the number $A$ is proportional to the volume
necessary for a single chemostat to have $s_{\rm ref}$ as resource
concentration at steady state (remind that this volume is equal to
$Q/\mu(s_{\rm ref})\equiv QA$). Therefore, a configuration
with a lateral-diffusive compartment would require a smaller volume than that of the single chemostat exactly
when the number $B$ is negative. Figure \ref{fig:Illustration} illustrates that this
is possible only when $s_{\rm ref}$ is above the minimizer $\hat s$ of the function $g$
(remind that the function $g$ is strictly convex, since it is equal to $1/\beta$ and
$\beta$ is strictly concave by Lemma \ref{lemma1}). \\ 
\begin{figure}[ht!]
\centering
\subfigure[$s_{\rm ref}>\hat s$ ($B$ negative)]{\includegraphics[width=0.25\textwidth]{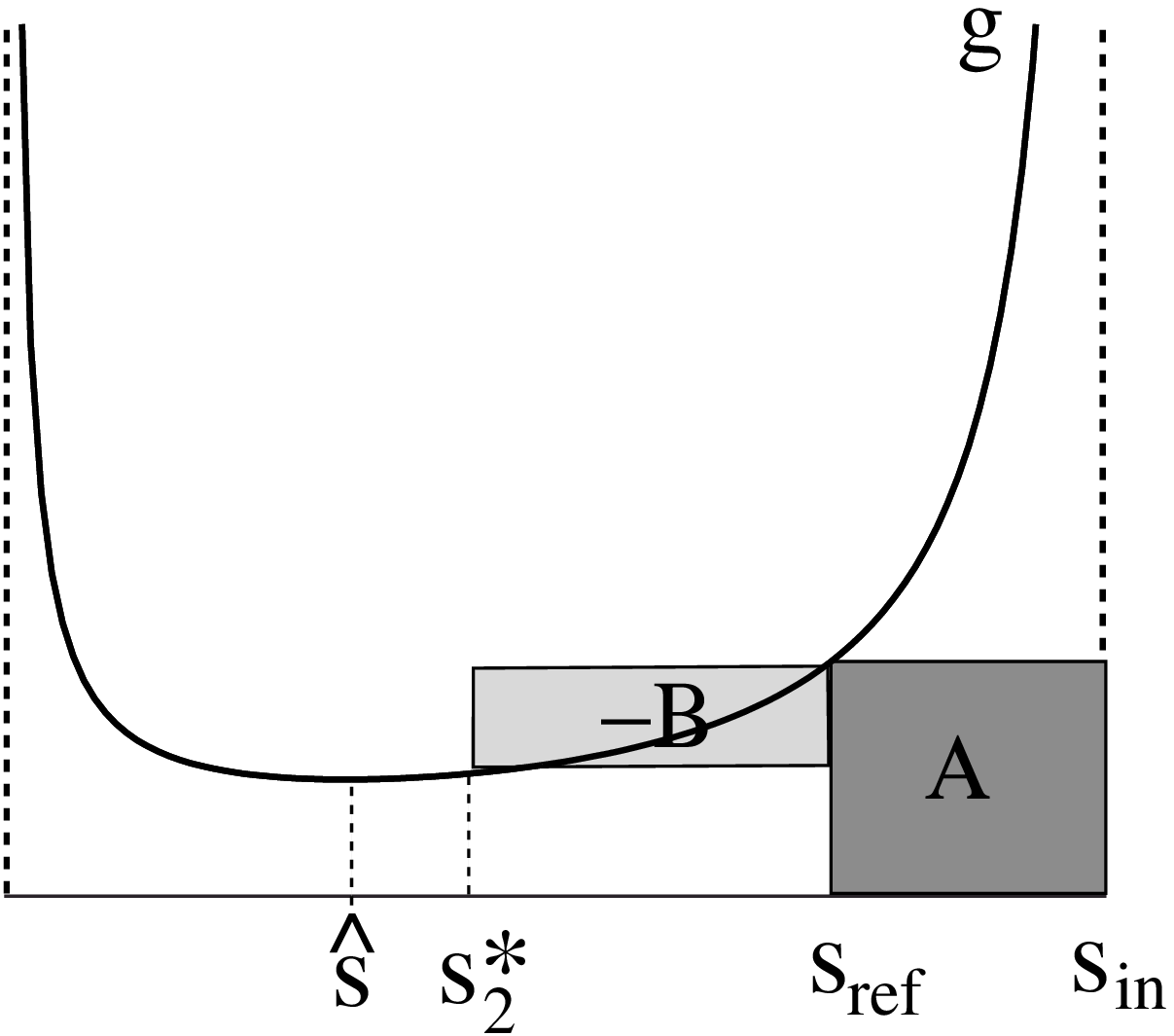}}\hspace{1cm}
\subfigure[$s_{\rm ref}<\hat s$ ($B$ positive)]{\includegraphics[width=0.25\textwidth]{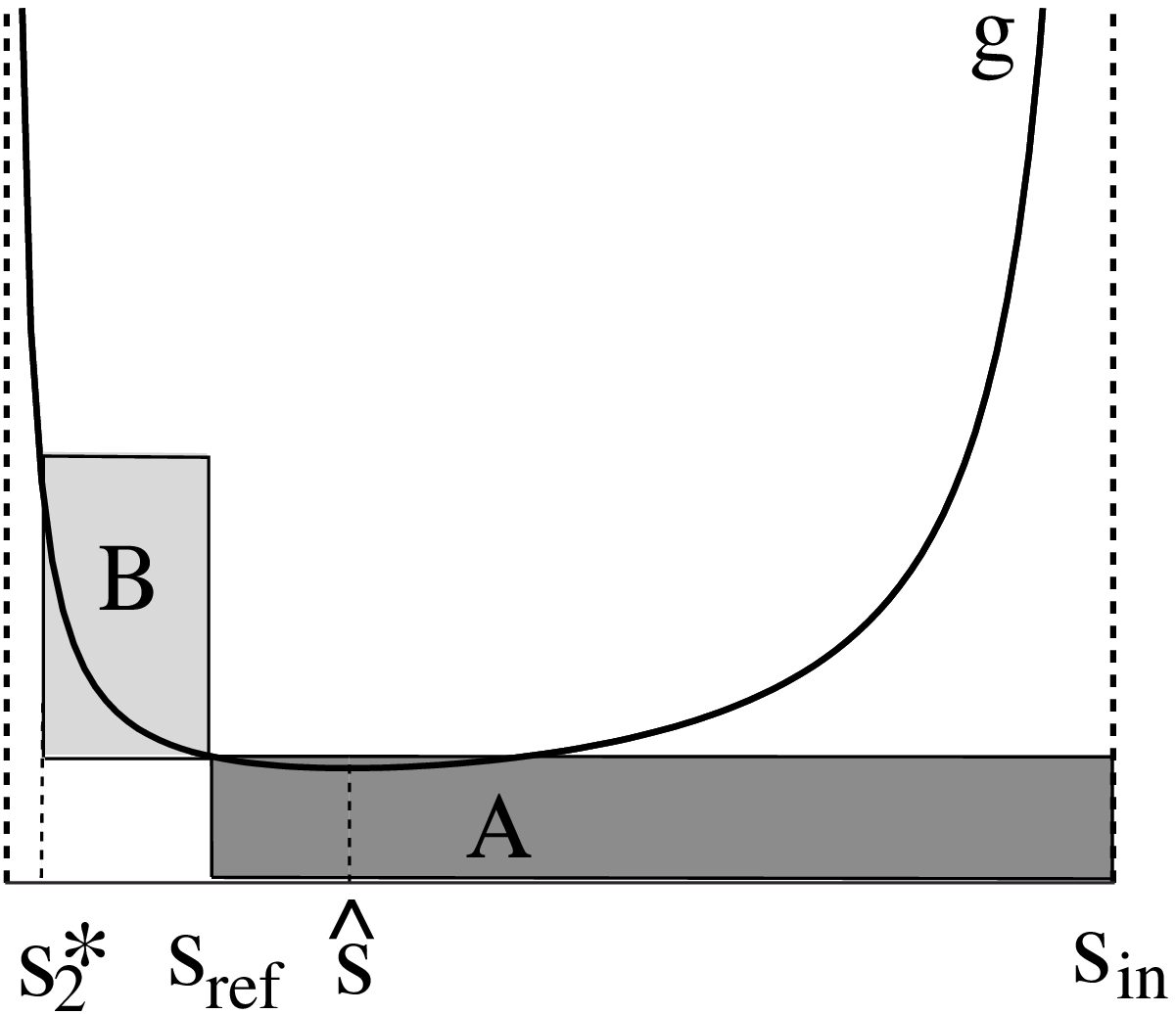}}
\caption{Graphical representation of quantities $A$ and $|B|$ in \eqref{eq:Vdiscussion}.}
\label{fig:Illustration}
\end{figure}
Furthermore, the quantity $B$ is equal to $G(s_{2}^\star)$, where the function $G$ defined in \eqref{defgG} admits an unique minimum at $s_{G} \in [0,s_{\rm ref}]$. Proposition \ref{prop:dgiven} states that, when $s_{\rm ref}>\hat{s}$, the optimal value of $s_{2}^{\star}$ (that is, the value of $s_{2}^{\star}$ which minimizes the total volume) is $s_{G}$ when $\alpha\leq s_{G}$ and $\alpha$ in other case, the later scenario corresponding to the particular configuration with $V_{1}=0$ (since $V_{1}=Qg(s_{\rm ref})(s_{\rm in}-s_{\rm ref}) + d g(s_{\rm ref})(s_{2}^\star-s_{\rm ref})$). A graphical interpretation of the optimized structures obtained when parameter $d$ is fixed is given in Figure\ref{fig:optimized}. When the diffusion rate can be tuned, the optimized configuration is as depicted in Figure \ref{fig:optimized}-(c).\\
 
Finally, let us recall from the theory of optimal design of chemostats in series that the first tank (when it is optimal to have more than one tank) has systematically a resource concentration $s_{1}^\star$ above $\hat s$ at steady state (see, e.g., \cite{GBBT96,HRT03}). Thus, for an industrial perspective, we can state that a lateral-diffusive compartment for the first tank of an optimal series of chemostat could systematically improve the performance of the overall process.
\begin{figure}[ht!]
\centering
\subfigure[$s_{\rm ref}<\hat s$]{\includegraphics[width=0.3\textwidth]{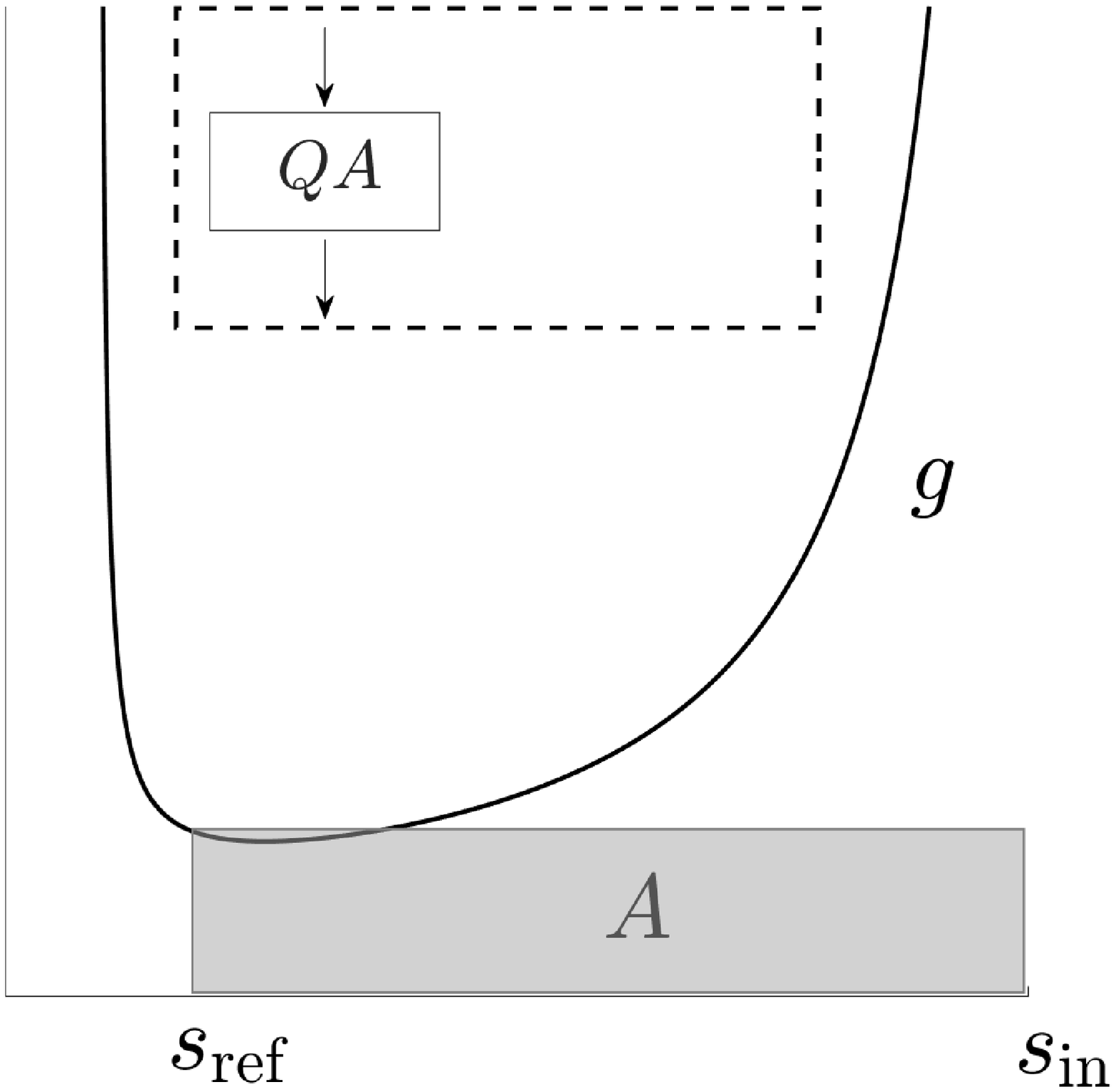}}\hspace{1cm}
\subfigure[$\alpha \leq s_{G}<\hat{s}<s_{\rm ref}$]{\includegraphics[width=0.3\textwidth]{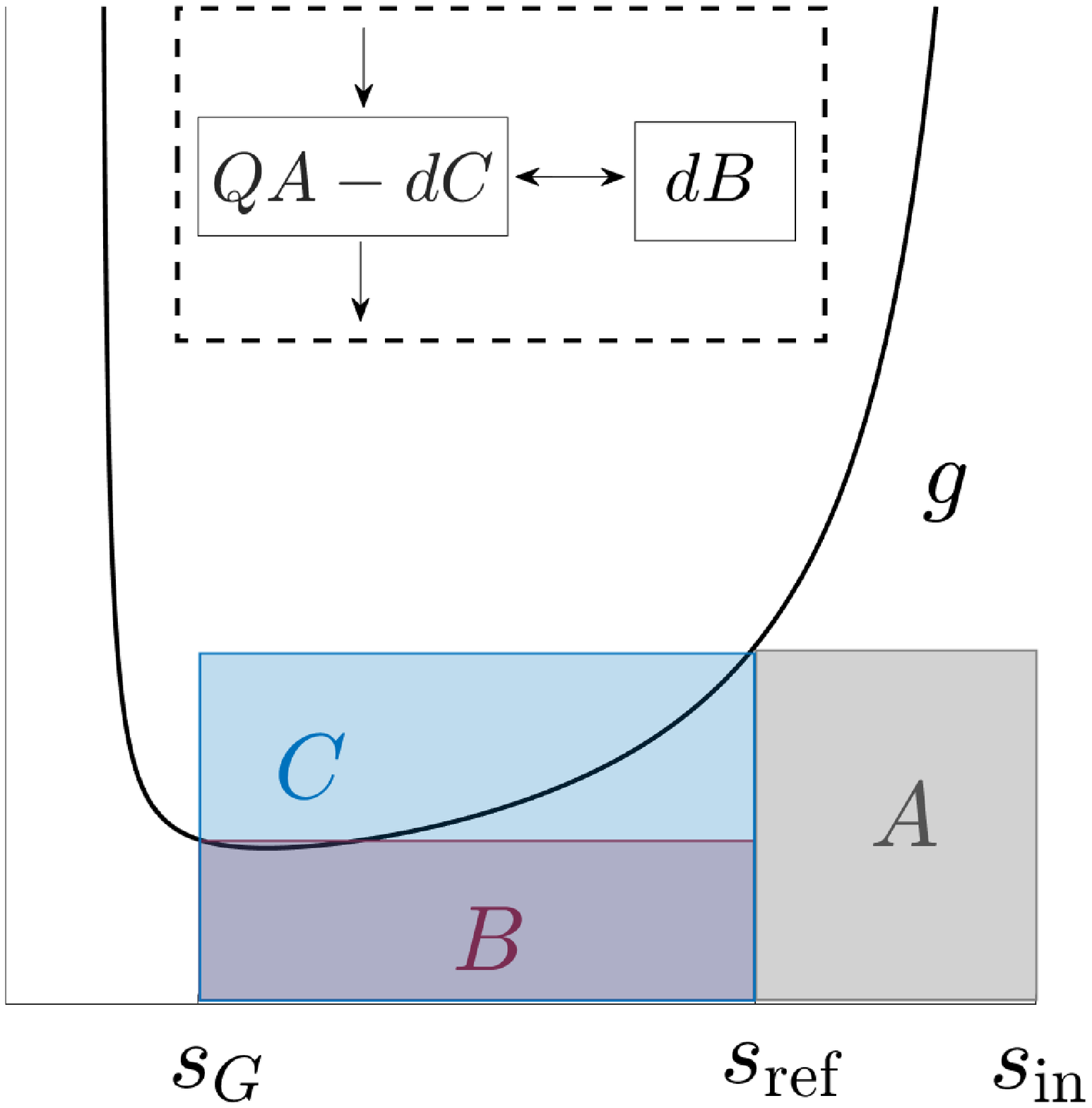}}
\subfigure[$\hat{s}\leq \alpha$ or $s_{G}<\alpha<\hat{s}<s_{\rm ref}$]{\includegraphics[width=0.3\textwidth]{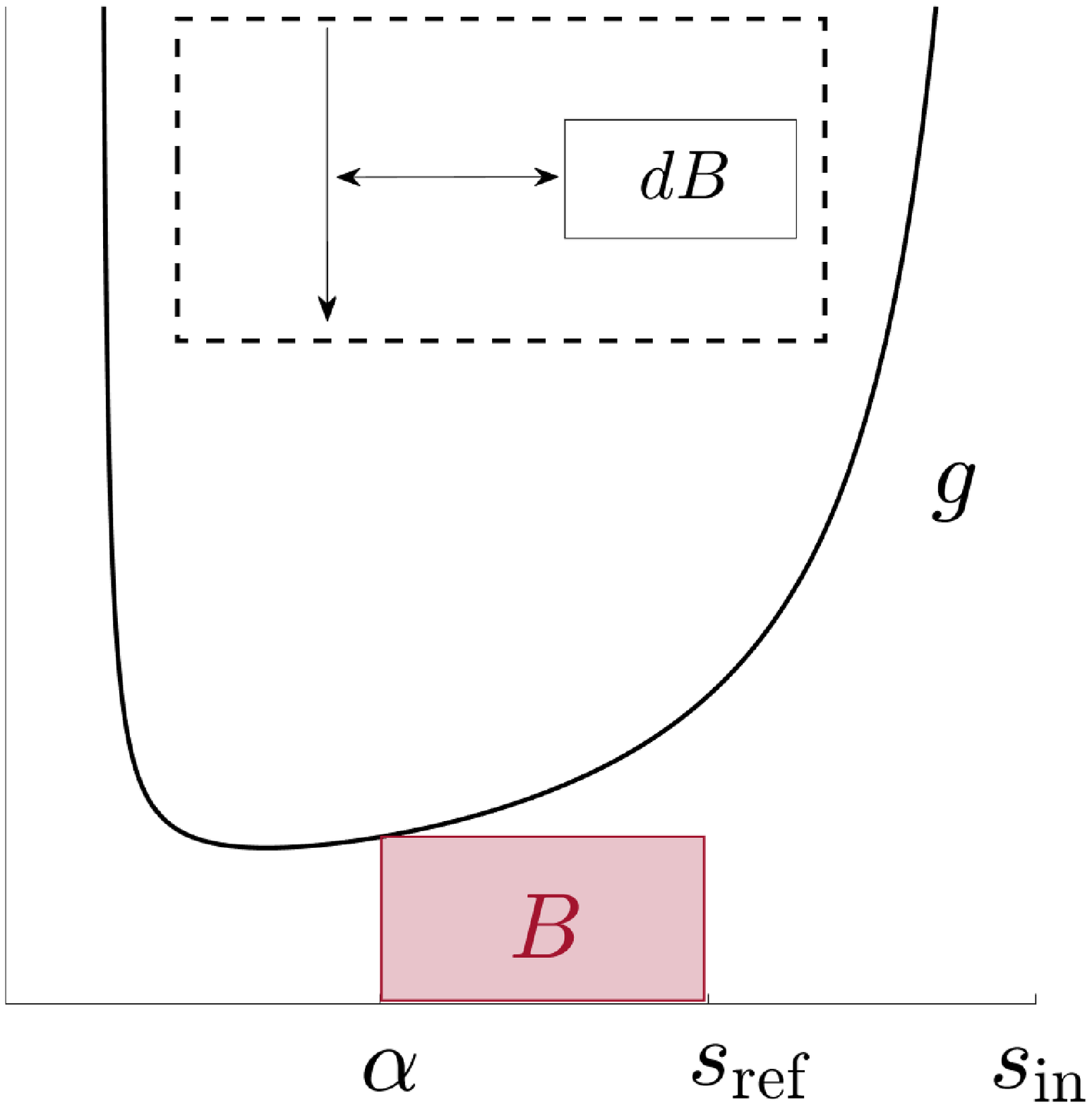}}
\caption{Graphical representation of the optimized configurations when parameter $d$ is fixed.}
\label{fig:optimized}
\end{figure}

\subsection{A numerical example}\label{sec:NumericalExample}
One may wonder how much could be gained (in terms of residence time) by using the proposed structure. Nevertheless, it is difficult to quantify the overall profit since the optimal design depends on parameters $s_{\rm in}$, $s_{\rm ref}$, $Q$, $\mu(\cdot)$ and $d$ (when it is not fixed beforehand). As an illustrative example, we compare the total optimal volumes $V^{\rm opt}(0)$, $V^{\rm opt}(Q)$ and $V^{\rm opt}(d^{\ast})$, obtained by solving problem \eqref{eq:OptDgiven} when $Q=1$, $s_{\rm in}=10$, $\mu(\cdot)$ is the Monod function with $\mu_{max}=1$ and $K=0.5$ (in this case, $\hat{s} \approx 1.79$) and diffusion coefficients $d=0$, $Q$ and $d^{\ast}$, respectively. More precisely, Figure \ref{fig:Influenced}-(a) compares the three diffusion coefficients (seen as functions of parameter $s_{\rm ref}$) while the associated optimal volumes are depicted in Figure \ref{fig:Influenced}-(b). Notice that $V^{\rm opt}(0)$ corresponds to the volume of the single-tank chemostat. 
\begin{figure}[ht!]
\centering 
\subfigure[Considered diffusion parameters (seen as functions of $s_{\rm ref}$).]{\includegraphics[width=0.4\textwidth]{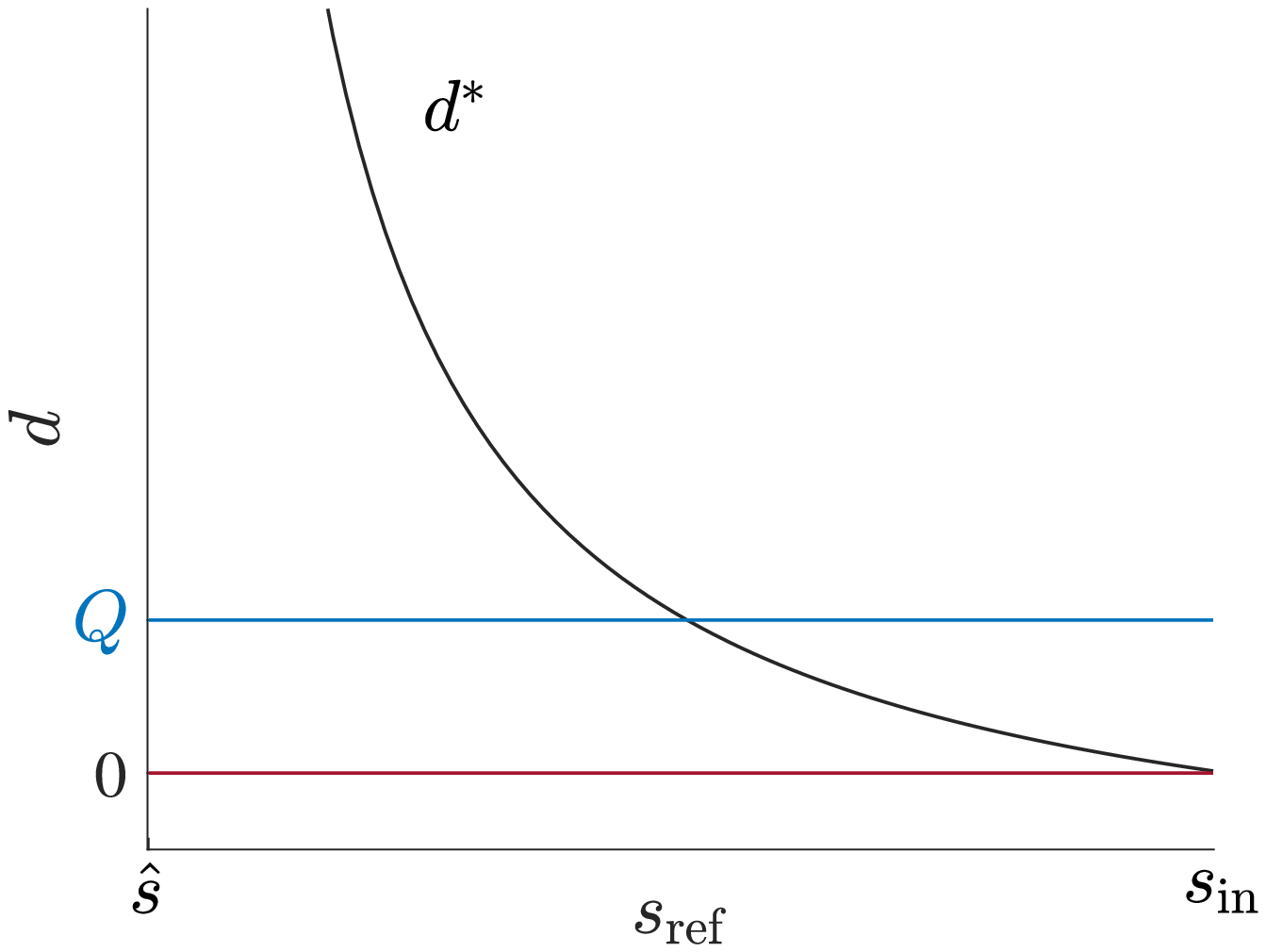}}\hspace{1cm}
\subfigure[Optimal volumes associated to the diffusion parameters in (a).]{\includegraphics[width=0.4\textwidth]{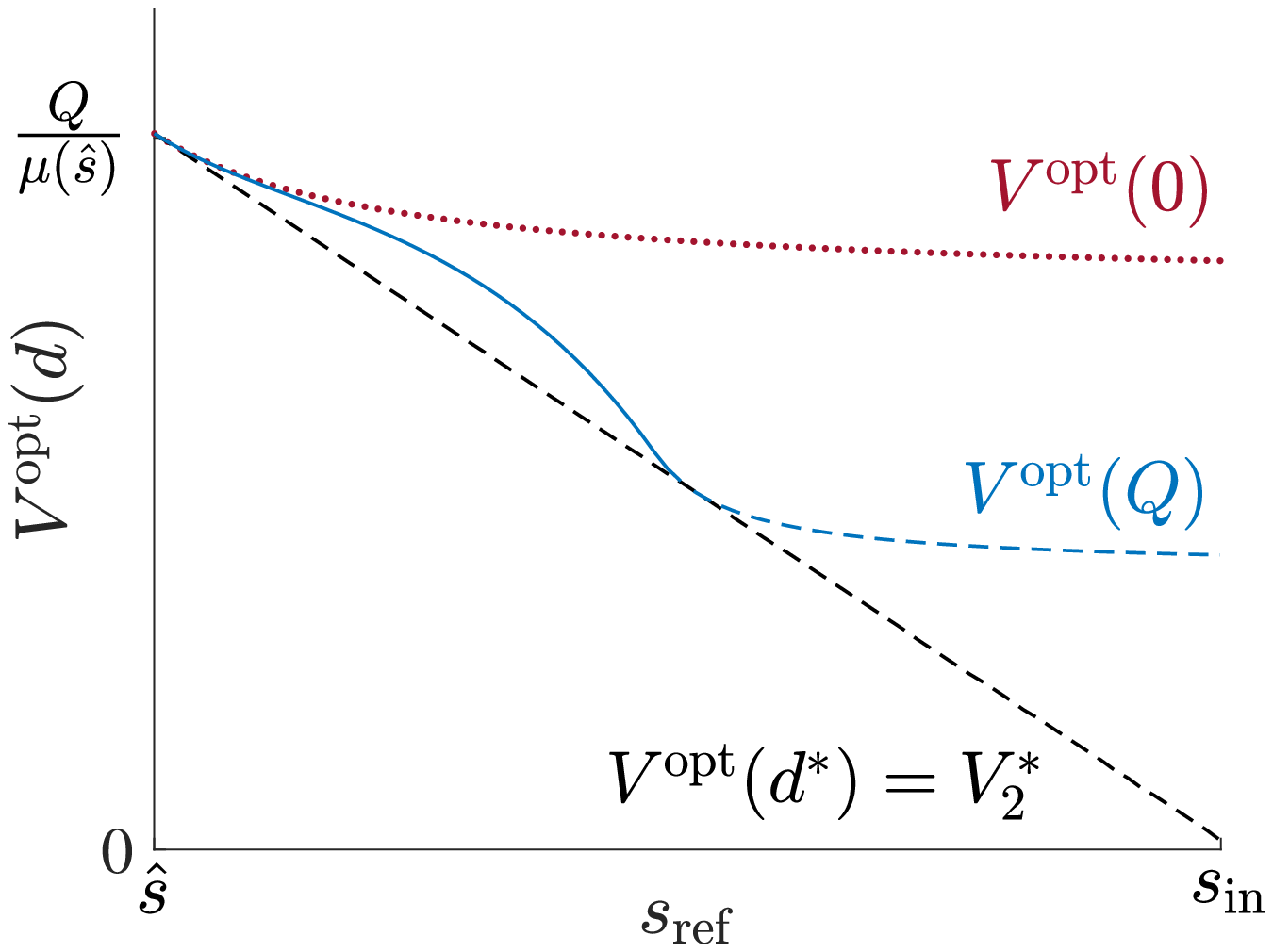}}
\caption{Comparison between the volumes $V^{\rm opt}(0)$, $V^{\rm opt}(Q)$ and $V^{\rm opt}(d^{\ast})$, obtained by solving problem \eqref{eq:OptDgiven} when $Q=1$, $s_{\rm in}=10$ and $\mu(\cdot)$ is the Monod function with $\mu_{max}=1$ and $K=0.5$ and diffusion coefficients $d=0$, $Q$ and $d^{\ast}$ ($d^{\ast}$ being the solution of problem \eqref{eq:OptDfree}), respectively. The solid, dashed and dotted lines in (b) represent, respectively, the values of $s_{\rm ref}$ for which the optimal design is composed of two tanks (of volumes $V_{1}^{\rm opt}$ and $V_{2}^{\rm opt}$), a single tank of volume $V_{2}^{\rm opt}$ or a single tank of volume $V_{1}^{\rm opt}$.}
\label{fig:Influenced}
\end{figure} \\
From Figure \ref{fig:Influenced}-(b) we remark that, when $d=Q$, there
exists a certain value of parameter $s_{\rm ref}$ in which the optimal
design transits from having two to one tank. Proposition
\ref{prop:OptDfree} infers that this transition occurs when $s_{\rm ref}=
\frac{s_{\rm in} + \hat{s}}{2}$ (in this case, when $s_{\rm ref} \approx
5.9$). One can observe that, for this particular value of $s_{\rm ref}$,
the single-tank volume is reduced approximately to its half and, in
general, the volume reduction becomes more significant as the value
$s_{\rm ref}$ increases. In those cases the gains are quite
significant.

\section{Conclusion}

In this work, we have identified situations for which a compartment
connected by ``lateral diffusion'' is beneficial for ecological
or engineering outcomes.

The analysis has first revealed two thresholds on the input
resource concentration which allow to distinguish three kinds of
situations for guaranteeing the existence of a positive equilibrium
depending on the diffusion rate $d$:
1. $d$ has to be positive but below a maximal value $\bar d$,
2. $d$ has simply to be positive,
3. $d$ can take any nonnegative value.
When the positive equilibrium exists, we have proved that it is
necessary asymptotically stable (although not always hyperbolic).

We have also studied the impact of the diffusion rate $d$ on the yield
conversion and concluded that there exists an optimal value of $d$
that maximizes the yield conversion (which is then necessarily better than
for a single tank) for ``extreme'' cases of the removal rate
(i.e. either small or large), which did not appear to be an intuitive result to us. 
This property implies that in natural habitats (such as in soil
ecosystems) the type of feeding (by convection or diffusion) from a
given flow rate could have a significant impact on the resource conversion.

In terms of total volume required for achieving a given yield conversion, we have
provided conditions that discriminate which configurations between one
or two tanks are the best.
Our conclusion is that a lateral compartment is beneficial compared to
a single chemostat when the yield conversion is not too important
(i.e. when the output resource concentration is not too small compared to the input one).
Surprisingly, we have also found that the limiting case of a
single tank purely connected by diffusion to the main stream (as depicted on
Figure \ref{figdilution}), and not crossed by the stream as in the classical chemostat, can provide the minimal volume.
For an industrial perspective, we can state that a lateral diffusive compartment for the first tank of an optimal series of chemostat
could systematically improve the performance of the overall process. Therefore, the analysis of combinations of series and lateral diffusive compartments (which is out of the scope of the present study) would most probably exhibit non-intuitive configurations that have not yet been considered in the literature. This would be the matter of a future work.

\section{Acknowledgments}
The authors thank the French LabEx Numev (project ANR-10-LABX-20) for
the postdoctoral grant of the First Author at MISTEA lab, Montpellier, France.

\bibliographystyle{plain}

\end{document}